\documentclass[12pt]{amsart}
\usepackage{times}
\DeclareMathAlphabet{\mathpzc}{OT1}{pzc}{m}{it}

\usepackage[T1]{fontenc}
\usepackage{dsfont}
\usepackage{mathrsfs}
\usepackage[colorlinks]{hyperref}
\usepackage{xcolor}
\usepackage[a4paper,asymmetric]{geometry}
\usepackage{mathscinet}
\usepackage{fullpage}
\usepackage{latexsym}
\usepackage{graphicx}
\usepackage{epstopdf}
\usepackage{amsthm}
\usepackage{amssymb}
\usepackage{amsfonts}
\usepackage{amsmath}

\newtheorem{theorem}{Theorem}[section]

\newtheorem{lem}[theorem]{Lemma}

\newtheorem{pro}[theorem]{Proposition}

\theoremstyle{definition}

\newtheorem{example}[theorem]{Example}
\newtheorem{rmk}[theorem]{Remark}
\newtheorem{thm}[theorem]{Theorem}

\newtheorem*{assA*}{Assumption A}

\numberwithin{equation}{section}

\newcommand{\e}{\epsilon}

\begin{document}

\title[LDP for weak approximation to Gaussian processes]
{Functional large deviations for Stroock's approximation to a class of Gaussian processes with application to small noise diffusions}

\author[H.~Jiang]{Hui Jiang}
\address[H.~Jiang]{School of Mathematics, Nanjing University of Aeronautics and Astronautics, China.}
\email{huijiang@nuaa.edu.cn}

\author[L.~Xu]{Lihu Xu}
\address[L.~Xu]{1. Department of Mathematics, Faculty of Science and Technology, University of Macau, Macau S.A.R., China. 2. Zhuhai UM Science \& Technology Research Institute, Zhuhai, China}
\email{lihuxu@um.edu.mo}

\author[Q.~Yang]{Qingshan Yang}
\address[Q.~Yang]{School of Mathematics and Statistics, Northeast Normal University, China.}
\email{yangqr66@gmail.com}


\makeatletter
\@namedef{subjclassname@2020}{%
  \textup{2020} Mathematics Subject Classification}
\makeatother


\begin{abstract}
Letting~$N=\left\{N(t), t\geq0\right\}$ be a standard Poisson process, Stroock~ \cite{Stroock-1981} constructed
a family of continuous processes by
$$\Theta_{\epsilon}(t)=\int_0^t\theta_{\epsilon}(r)dr, \ \ \ \ \ 0 \le t \le 1,$$
where $\theta_{\epsilon}(r)=\frac{1}{\epsilon}(-1)^{N(\epsilon^{-2}r)}$, and proved that it weakly converges to a standard Brownian motion under the continuous function topology.
We establish the functional large deviations principle (LDP) for the approximations of a class of Gaussian processes constructed by integrals over $\Theta_{\epsilon}(t)$, and find the explicit form for rate function.

As an application, we consider the following (non-Markovian) stochastic differential equation
\begin{equation*}
\begin{aligned}
X^{\epsilon}(t)
&=x_{0}+\int^{t}_{0}b(X^{\epsilon}(s))ds+\lambda(\epsilon)\int^{t}_{0}\sigma(X^{\epsilon}(s))d\Theta_{\epsilon}(s),
\end{aligned}
\end{equation*}
where $b$ and $\sigma$ are both Lipschitz functions, and establish its Freidlin-Wentzell type LDP as $\epsilon \rightarrow 0$. The rate function reads as
\begin{equation*}
\begin{aligned}
\mathcal{J}(x)=&\inf\left\{\int^{1}_{0}\Gamma(\phi(r))dr; ~\phi\in L^2([0,1]) {\rm \ is \ such \ that \ } \right.\\
&\left.\quad x(t)=x_{0}+\int^{t}_{0}b(x(s))ds+\int^{t}_{0}\sigma(x(s))\phi(s)ds \right\}.
\end{aligned}
\end{equation*}
where (1). if $\lim_{\epsilon \rightarrow 0} \frac{\lambda(\epsilon)}{\epsilon}=\kappa>0$,
\begin{equation*}
\begin{aligned}
\Gamma(x)=\left\{\begin{array}{ll}
1-\sqrt{1-\frac{x^{2}}{\kappa^{2}}}, & ~|x|\leq \kappa,\\
+\infty, & |x|>\kappa;\\\end{array}\right.
\end{aligned}
\end{equation*}
 (2). if $\lim\limits_{\epsilon\to0}\lambda(\epsilon)=0, \lim\limits_{\epsilon\to0}\frac{\lambda(\epsilon)}{\epsilon}=+\infty,$
$$\Gamma(x)=\frac{x^{2}}{2},  \ \ \ \ x\in\mathbb R.$$
The rate function clearly indicates a phase transition phenomenon as $\lambda(\epsilon)$ moves from one region to the other.   
\end{abstract}

\maketitle

\section{Introduction}

Letting~$N=\left\{N(t), t\geq0\right\}$ be a standard Poisson process, Stroock~(\cite{Stroock-1981}) constructed
a family of continuous process~
$\Theta_{\epsilon}=\left\{\Theta_{\epsilon}(t), 0\leq t\leq1\right\}$ as the following:
$$
\Theta_{\epsilon}(t)=\int_0^t\theta_{\epsilon}(r)dr, \quad \theta_{\epsilon}(t)=\frac{1}{\epsilon}(-1)^{N(\epsilon^{-2}t)},\quad 0 \leq t\leq1,
$$
and showed that it weakly converges to a standard Brownian motion~$W=\left\{W(t), 0\leq t\leq1\right\}$
in the space~$C\big([0,1]\big)$.
There have been a lot of researches following the ideas in Stroock's pioneering work.  Delgado and Jolis \cite{DJ-2000} constructed the following Gaussian process:
\begin{equation}\label{def-Y-epsilon}
G_{\epsilon}(t)
=\int_0^1K(t,r)\theta_{\epsilon}(r)dr,
\quad 0\leq t\leq1,
\end{equation}
where $K:[0,1] \times [0,1] \rightarrow \mathbb R$ is a deterministic kernel function, and
 proved that $G_\epsilon(t)$ weakly converges in $C([0,1])$ to
\begin{equation}\label{def-Y}
G(t)=\int_0^1K(t,r)dW(r), \ \ \ 0\le t \le 1.
\end{equation}
Jiang and Yang \cite{JY-2021} considered the functional moderate deviations associated to this convergence.
For more related works, we refer the reader to \cite{BJ-2000,  WYY-2013} for the weak approximation of the Brownian sheet from a Poisson process in the plane or random walk,
to \cite{DL-2011} for the approximation of fractional Brownian motion,
and to \cite{BBR-2016, BJQ-2010, BMQ-2020, BNRT-2010, DJQ-2013} for the approximation of stochastic (partial) differential equations.

To the best of our knowledge, there are no works on LDP for these Stroock type approximations. We study in this paper the functional LDP of $G_{\epsilon}$ and Freidlin-Wentzell type LDP for stochastic differential equation (SDE) driven by $\Theta_{\epsilon}$. The first main result, Theorem \ref{thm-main-result} below, shows that $(\e G_{\epsilon}(t))_{0 \le t \le 1}$ satisfies a functional LDP with a rate function quite different from the classical Schilder type \cite[Chapter 5]{DZ-1998}.
More generally and importantly, we establish in Theorem \ref{thm-multiplicative noise} below Freidlin-Wentzell type LDP for SDE driven by $\Theta_{\epsilon}$, as the parameter $\lambda(\e)$ moves from one regime to the other, the rate function indicates an interesting phase transition.

Theorem \ref{thm-main-result} is proved by a standard method in LDP, discretizing $G_{\e}(t)$ and showing the LDP for the discretizaiton and  {the exponential tightness in continuous function topology}, while the proof of Theorem \ref{thm-multiplicative noise} is much more involved and the difficulties are
summarized as the following four aspects. (i) Because the joint process $(X^\e(t), \theta_\e(t))$ is Markov but its marginal $X^\e(t)$ is not, it seems difficult to apply the well known weak convergence method for LDP developed by Budhirija and Dupuis \cite{Bu-Du-2000}.  (ii) In order to handle the driven noise $\Theta_{\epsilon}$, we decompose it into a martingale and a remainder. To bound the terms related to this remainder,
we need to carefully analyze the behavior of $\theta_\e$ and use the properties of the drift $b$ and diffusion coefficient $\sigma$. (iii) Due to fast switches between $-1$ and $+1$, we use a Gaussian integrability type result, see for instance {\cite[Condition A18 and Theorem A19]{FV-2010} }to prove exponential tightness of $(X^\e(t))_{0 \le t \le 1}$.  Thanks to this exponential tightness, rather than the standard method of discretizing SDE, we directly compare the SDE with a stochastic process $\hat X^\e$. In this comparison, we need to estimate several small ball probabilities, in which stopping times play a crucial role in the `divide and conquer' procedure. (iv) The developed concentration inequalities in Section \ref{s:a} has independent interests and can hopefully be used in future.
 \vskip 3mm


Throughout this paper the deterministic kernel $K: [0,1]\times[0,1]\to\mathbb R$ satisfies the following conditions:

{\bf (H1)} $K$ is measurable and $K(0,r)=0$ for any $r\in[0,1]$.

{\bf (H2)} There exists a strictly increasing and continuous function $\mathcal{G}:[0,1]\to\mathbb R$ and a constant $\alpha>0$ such that for any $0\leq s<t\leq1$,
\begin{equation*}
\begin{aligned}
\int^{1}_{0}(K(t,r)-K(s,r))^{2}dr\leq(\mathcal{G}(t)-\mathcal{G}(s))^{\alpha}.
\end{aligned}
\end{equation*}
 It is easy to check that the typical Gaussian processes such as Ornstein-Uhlenbeck (OU) process and fractional Brownian motion satisfy the conditions {\bf (H1)} and {\bf (H2)}. More precisely,  if $K(t,r)=e^{\lambda(r-t)}$ with $\lambda>0$, $G$ is an Ornstein-Uhlenbeck process; if $K(t,r)=K^{H}(t,r)$ defined as follows

\begin{equation}\label{FBM-kernal}
\begin{aligned}
K^{H}(t,r)=\left\{\begin{array}{ll}
c_{H}r^{\frac{1}{2}-H}1_{[0,t]}(r)\int_r^t(u-r)^{H-\frac{3}{2}}u^{H-\frac{1}{2}}du,& \frac{1}{2}<H<1;\\
\tilde{c}_{H}1_{[0,t]}(r)\left(\big(t/r\big)^{H-\frac{1}{2}}(t-r)^{H-\frac{1}{2}}\right.\\
\left.-\left(H-\frac{1}{2}\right)r^{\frac{1}{2}-H}\int_r^t(u-r)^{H-\frac{1}{2}}u^{H-\frac{3}{2}}du\right),& 0<H<\frac{1}{2},\\\end{array}\right.
\end{aligned}
\end{equation}

where
$$
c_H=\left(\frac{H(2H-1)}{\beta\left(2-2H, H-\frac{1}{2}\right)}\right)^{1/2},~
\tilde{c}_H=\left(\frac{2H}{(1-2H)\beta\left(1-2H, H+\frac{1}{2}\right)}\right)^{1/2},
$$
the process ~$G$ is a fractional Brownian motion with Hurst index~$H$ ~(Chapter~5 in ~\cite{Nualart}).

For a closed interval $E \subset \mathbb R$, we denote by $C(E)$ the set of all continuous functions from $E$ to $\mathbb R$, by $C^1(E)$ the set of all functions $f$ such that $f$ and $f'$ are both bounded continuous, and by $L^2(E)$ the set of all functions $f$ such that $\int_E |f(x)|^2 dx<\infty$. For any bounded measurable function $f: E \rightarrow \mathbb R$, define
$\|f\|=\sup_{x \in E} |f(x)|$. Furthermore, we define $|x|=\sqrt{x_1^2+...+x_k^2}$ for any $x \in \mathbb R^k$.

The remainder of this paper is organized as follows.
In Section \ref{s:2}, we will give the two main results, one being proved in Section \ref{s:3} and the other one in Section \ref{s:4}.  Section \ref{s:a} is the appendix which includes the proofs of the functional LDP for
$\epsilon \Theta_{\epsilon}$ and the auxiliary estimates to be used in Section \ref{s:4}.

\section{Main results} \label{s:2}


Our first main result is about the large deviations of $\{\epsilon G_{\epsilon}(\cdot),\epsilon>0\}$ at the scale $\epsilon$. 
\begin{thm}\label{thm-main-result}
Under assumptions~{\bf (H1)} and~{\bf (H2)}, the family~$\{\epsilon G_{\epsilon}(\cdot),\epsilon>0\}$ satisfies the functional large deviation in $C([0,1])$,
with speed~$\epsilon^{2}$ and rate function ~$I$ defined for ~$f\in C([0,1])$
\begin{equation}\label{rate function-main result}
I(f)=\inf\left\{\int^{1}_{0}\Lambda^{*}(\phi(r))dr; ~f(t)=\int^{1}_{0}K(t,r)\phi(r)dr,~\phi\in L^2([0,1]),~t\in[0,1]\right\},
\end{equation}
where
\begin{equation}\label{def-rate function-1}
\Lambda^{*}(x)=\left\{\begin{array}{ll}
1-\sqrt{1-x^{2}}, & ~|x|\leq1;\\
+\infty, & |x|>1.\\\end{array}\right.
\end{equation}
Explicitly, for any Borel measurable set~$A\subseteq C([0,1])$,
$$
-\inf_{f\in A^{\circ}}I(f)\leq
\limsup_{\epsilon\to 0}\epsilon^{2}\log \mathbb{P}\left(\epsilon G_{\epsilon}(\cdot)\in A\right)
\leq-\inf_{f\in \overline{A}}I(f).
$$
\end{thm}




\begin{rmk}
Under the assumptions~{\bf (H1)} and~{\bf (H2)}, for any fixed~$t\in[0,1]$, $K(t, \cdot)$ belongs to $L^{2}([0,1])$ and is not necessarily differentiable
or even continuous and bounded. Therefore, integration by parts formula may be not available for ~$G_{\epsilon}(\cdot)$.
To overcome this difficulity, we will use the exponentially good approximation
in large deviations theory \cite[Section 4.2.2]{DZ-1998}), exponential martingale technique, and Besov-L\'evy modulus embedding \cite{FV-2010}.
\end{rmk}

Here we only give three typical examples: Brownian motions, fractional Brownian motions and Ornstein-Uhlenbeck processes.  We can construct as many examples as we wish, which satisfy {\bf (H1)} and {\bf (H2)}.
\begin{example}\label{representation2}
 {\bf 1. (Brownian motions)}
Let $K(t,r)=\xi(t)\eta(r)1_{[0,t]}(r), ~t,r\in[0,1]$.
 Here, $\xi$ and $\eta$ are measurable functions satisfying $\xi\in C([0,1])$, $\eta\in L^{2}([0,1])$
 and $\mathcal{L}(\{\xi=0\})=0$ where $\mathcal{L}$ denotes Lebesgue measure.  We have
\begin{equation*}
\begin{aligned}
I(f)=\left\{\begin{array}{ll}
\int^{1}_{0}\Lambda^{*}\left((\frac{f}{\xi})^{'}(r)\frac{1_{\{\eta\neq0\}}(r)}{\eta^{2}(r)}\right)dr,
 &\textrm{if }\ \frac{f}{\xi}\ \textrm{is absolutly continuous and}
 \left(\frac{f}{\xi}\right)^{'}\frac{1_{\{\eta\neq0\}}}{\eta}\in L^2([0,1]);\\
 +\infty,& \textrm{otherwise. }\\\end{array}\right.
\end{aligned}
\end{equation*}
When $\xi(t) \equiv 1$ and $\eta(r) \equiv 1$, then $\Theta_{\epsilon}$ converges to Brownian motion and the corresponding
LDP rate function is
\begin{equation*}
\begin{aligned}
I(f)=\left\{\begin{array}{ll}
\int^{1}_{0}\Lambda^{*}(f^{'}(r))dr,
 &\textrm{if } f \ \textrm{is absolutly continuous and}~
f' \in L^2([0,1]);\\
 +\infty,& \textrm{otherwise.}\\\end{array}\right.
\end{aligned}
\end{equation*}

{\bf 2. (Fractional Brownian motions)} For~$K^{H}(t,r)$ given by~(\ref{FBM-kernal}),
$G(t)=\int_0^1K^{H}(t,r)dW(r)$ is a fractional Brownian motion with Hurst index~$H$.
Applying Theorem 2.1 in Decreusefond and \"Ust\"unel~(\cite{DU-1999}),
the mapping $\mathbb S$ there is an injection.
Moreover, {\bf{(H1)}} and {\bf{(H2)}} both hold with $\mathcal G(t)=Ct, C>0$, $\alpha=2H$~(\cite{DJ-2000}).
The rate function ~$I$ is given by
$$
I(f)=\left\{\begin{array}{ll}
\left\|\Lambda^{*}(\mathbb S^{-1}f)\right\|^{2}_{L^{2}([0,1])}, & ~\textrm{if}~f\in\mathbb S(L^{2}([0,1]));\\
+\infty,&~\textrm{otherwise}.\\\end{array}\right.
$$

{\bf 3. (Ornstein-Uhlenbeck processes)} For the kernel~$K(t,r)=Ce^{\lambda(r-t)}1_{(0,t]}(r)$ with~$\lambda>0$,  we have ~(H1) and~(H2) both hold with $\mathcal G(t)=Ct, C>0$~(\cite{DJ-2000}). We have $ G_t=\int_0^1Ce^{\lambda(r-t)}dW(r)$.
By Theorem~\ref{thm-main-result} and Corollary~\ref{representation2},
the rate function ~$I$ is defined by
\begin{equation*}
\begin{aligned}
I(f)=\left\{\begin{array}{ll}
\int^{1}_{0}\Lambda^{*}\left(\frac{f'(r)+\lambda f(r)}{C}\right)dr, &\textrm{if }f\textrm{ is absolutly continuous and}~\int^{1}_{0}|f'(r)|^{2}dr<+\infty;\\
+\infty,&~\textrm{otherwise}.\\\end{array}\right.
\end{aligned}
\end{equation*}
\ \\
\end{example}

Stimulated by Freidlin-Wentzell type LDP, since we have proved the functional LDP for the convergence of $\Theta_{\epsilon}$ to Brownian motion, it is natural for us to consider the small noise LDP for SDEs whose driven force is $\Theta_{\epsilon}$.
\begin{equation}\label{def-sde-FW2}
\begin{aligned}
X^{\epsilon}(t)
&=x_{0}+\int^{t}_{0}b(X^{\epsilon}(s))ds+\lambda(\epsilon)\int^{t}_{0}\sigma(X^{\epsilon}(s))d\Theta_{\epsilon}(s)
\end{aligned}
\end{equation}
where~$\Theta_{\epsilon}(t)=\int_{0}^{t} \theta_{\epsilon}(s) ds$ with $\theta_{\epsilon}(t)=\frac{1}{\epsilon}(-1)^{N(\epsilon^{-2}t)}$, and $\lambda(\epsilon) \rightarrow 0$ as $\epsilon \rightarrow 0$. We further assume that

{\bf (H3)} For some ~$L\in(0,+\infty)$, we have for all $z,z'\in\mathbb R$,
\begin{equation*}
|b(z)-b(z')|+|\sigma(z)-\sigma(z')|\leq L|z-z'|.
\end{equation*}




Our second main result, which is more interesting than the first one, is the following Freidlin-Wentzell type LDP, and it indicates an interesting phase transition phenomenon for the rate functions as $\lambda(\epsilon)$ moves from one region to the other.
\begin{thm}\label{thm-multiplicative noise}
{Assume {\bf (H3)} holds}. For any $x\in C([0,1])$, define
\begin{equation}\label{rate function-sde}
\begin{aligned}
\mathcal{I}(x)=&\inf\left\{\int^{1}_{0}\Gamma(\phi(r))dr; ~\phi\in L^2([0,1]) \ {\rm is \ such \ that \ } \right.\\
&\left.\quad x(t)=x_{0}+\int^{t}_{0}b(x(s))ds+\int^{t}_{0}\sigma(x(s))\phi(s)ds \right\}.
\end{aligned}
\end{equation}

\noindent (1). If $\lambda(\epsilon)$ satisfies
\begin{equation}\label{def-kappa}
\lim_{\epsilon \rightarrow 0} \frac{\lambda(\epsilon)}{\epsilon}=\kappa>0,
\end{equation}
 then $\left\{X^{\epsilon}(t), ~t\in[0,1]\right\}$ satisfies the large deviations with speed $\lambda^2(\epsilon)$ and rate function $\mathcal{I}$ with
\begin{equation*}
\begin{aligned}
\Gamma(x)=\left\{\begin{array}{ll}
1-\sqrt{1-\frac{x^{2}}{\kappa^{2}}}, & ~|x|\leq \kappa;\\
+\infty, & |x|>\kappa.\\\end{array}\right.
\end{aligned}
\end{equation*}

\noindent (2). If ~$\lambda(\epsilon)$ satisfies
$$
\lim\limits_{\epsilon\to0}\lambda(\epsilon)=0, ~~ \ \ \ \ \ \lim\limits_{\epsilon\to0}\frac{\lambda(\epsilon)}{\epsilon}=+\infty,
$$
then $\left\{X^{\epsilon}(t), ~t\in[0,1]\right\}$ satisfies the large deviations with speed ~$\lambda^2(\epsilon)$ and rate function $\mathcal I$ with
$$\Gamma(x)=\frac{x^{2}}{2},  \ \ \ \ x\in\mathbb R.$$

\end{thm}

\begin{rmk}

As $\lim_{\epsilon\to0}\frac{\lambda(\epsilon)}{\epsilon}=0$, the LDP is trivial. In fact, for any $a\in\mathbb R$, a simple computation yields
\begin{equation*}
\lim_{\epsilon\to0}\lambda^2(\epsilon)\log\mathbb E\exp\left\{\frac{a}{\epsilon\lambda(\epsilon)}\int^{t}_{0}(-1)^{N(\epsilon^{-2}u)}du\right\}\leq\lim_{\epsilon\to0}\frac{|a|\lambda(\epsilon)}{\epsilon}=0, ~t\in[0,1].
\end{equation*}
By the standard argument of large deviation principle \cite{DZ-1998}, we can show that the sample path of $\left\{\frac{\lambda(\epsilon)}{\epsilon}\int^{t}_{0}(-1)^{N(\epsilon^{-2}u)}du, t\in[0,1]\right\}$ obeys the same large deviations as the functional $0$, thus the LDP with respect to \eqref{def-sde-FW2} is trivial.
\end{rmk}
\begin{rmk}
Because the joint process $(X^\e(t), \theta_\e(t))$ is Markov but the marginal process $X^\e(t)$ is not, it seems difficult to apply the well known weak convergence method for LDP developed by Budhirija and Dupuis \cite{Bu-Du-2000}, see more results in \cite{BDM-2011,WXZZ-2016} and the references therein for this weak convergence approach.  
\end{rmk}

\section{Proof of Theorem~\ref{thm-main-result}} \label{s:3}
We shall use the following standard strategy in the large deviation theory to prove Theorem~\ref{thm-main-result}:

(i) Show the LDP for finite dimensional distributions
$$
\mathbb{P}\left(\epsilon\Big(G_{\epsilon}(t_1),\cdots,G_{\epsilon}(t_k)\Big)^{\tau}\in\cdot\right),\quad~0<t_1<\cdots<t_k\leq1, k\geq1,
$$
where~$x^{\tau}$ denotes the transpose of ~$x\in\mathbb{R}^k$, ~$k\in\mathbb{N}$.

(ii) Show the following exponential tightness: for any~$L>0$, there exists a compact set~$K_L\subset C([0,1])$ such that
$$
\limsup_{\epsilon\to0}\epsilon^{2}\log\mathbb P\left(\epsilon G_{\epsilon}\in K^{c}_{L}\right)\leq-L.
$$
See Theorem 4.2.4 and Theorem 4.6.1 in Dembo and Zeitouni~(\cite{DZ-1998}) for more details.


\subsection{Functional LDP of $\epsilon \Theta_{\epsilon}$} In order to use the above standard strategy, we shall first establish the functional LDP for $\epsilon \Theta_{\epsilon}$. On the one hand, in the proof of (i), we shall apply the contract principle with $\epsilon \Theta_{\epsilon}$; on the other hand, the proof of the functional LDP for $\epsilon \Theta_{\epsilon}$ is helpful to understand that of Theorem~\ref{thm-main-result}.

Rewrite
\begin{equation}\label{Main term pro}
\begin{aligned}
\epsilon \Theta_{\epsilon}(t)=\epsilon^{2}\int^{\epsilon^{-2}t}_{0}\xi(r)dr,
\end{aligned}
\end{equation}
where
$$
\xi(t):=(-1)^{N(t)}, \quad t\in[0,1],
$$
is a ~$\{-1,1\}$-valued reversible Markov process.
We shall often use this formulation of $\epsilon \Theta_{\epsilon}(t)$ below.
\begin{pro}\label{functional LDP}
The family $\{\epsilon \Theta_{\epsilon}(\cdot),\epsilon>0\}$ satisfies the functional large deviation in $C([0,1])$ as $\epsilon \rightarrow 0$ ,
with speed~$\epsilon^{-2}$ and rate function ~$J$ defined as: for ~$\varphi\in C([0,1])$
\begin{equation}\label{rate function-main result}
J(\varphi)=\inf\left\{\int^{1}_{0}\Lambda^{*}(\phi(r))dr; ~\varphi(t)=\int^{t}_{0} \phi(r)dr\in \mathcal{AC},~t\in[0,1]\right\},
\end{equation}
where
\begin{equation*}
\Lambda^{*}(x)=\left\{\begin{array}{ll}
1-\sqrt{1-x^{2}}, & ~|x|\leq1;\\
+\infty, & |x|>1\\\end{array}\right.
\end{equation*}
and~$\mathcal{AC}=\Big\{f; f~\textrm{~is absolutly continuous on}~[0,1], \int^{1}_{0}|f'(r)|^{2}dr<+\infty\Big\}$.
\end{pro}
The proof is standard and thus given in the appendix.

\subsection{Proof of Theorem \ref{thm-main-result}}
\subsubsection{Finite dimensional large deviations}
In this subsection, for fixed ~$0<t_1<\cdots<t_k\leq1$,
we will study the large deviations for~$\epsilon\Big(G_{\epsilon}(t_1),\cdots,G_{\epsilon}(t_k)\Big)^{\tau}$.
Since $K(t_i,\cdot)\in L^{2}([0,1])$, there exists a sequence of finite variational functions
~$\Big\{K_{m}(t_{i},\cdot), m\in\mathbb N\Big\}$ such that
$$
\lim_{m\to\infty}\max_{1\leq i\leq k}\int_0^1\big|K(t_i,r)-K_{m}(t_i,r)\big|^2dr=0.
$$
Define
\begin{equation}\label{def-Y-epsilon}
\begin{aligned}
G_{\epsilon}(t_1,\cdots, t_k)
&=\Big(\int^{1}_{0}K(t_1,r)\theta_{\epsilon}(r)dr,\cdots, \int^{1}_{0}K(t_k,r)\theta_{\epsilon}(r)dr\Big)^{\tau}\\
&=\Big(G_{\epsilon}(t_1),\cdots,G_{\epsilon}(t_k)\Big)
\end{aligned}
\end{equation}
and
\begin{equation}\label{def-Ym-epsilon}
\begin{aligned}
G_{m,\epsilon}(t_1,\cdots, t_k)
&=\left(\int^{1}_{0}K_{m}(t_1,r)\theta_{\epsilon}(r)dr,\cdots, \int^{1}_{0}K_{m}(t_k,r)\theta_{\epsilon}(r)dr\right)^{\tau}\\
&=\Big(G_{m,\epsilon}(t_1),\cdots,G_{m,\epsilon}(t_k)\Big).
\end{aligned}
\end{equation}
We will show in Lemma \ref{lem-Y-Ym} below that $G_{m,\epsilon}(t_1,\cdots, t_k)$ is an exponentially good approximation of $G_{\epsilon}(t_1,\cdots, t_k)$
~(Section 4.2.2 in ~\cite{DZ-1998}).
To do this, the following lemma plays a crucial role in our analysis.

\begin{lem}\label{lem-exponential estimation}
(\cite[Corollary 3.1]{JY-2021}) For any ~$0\leq\eta<\frac{1}{2}$ and ~$f\in L^{2}([0,1])$, it holds that
\begin{equation*}
\sup_{f\in L^{2}([0,1]),\|f\|_{L^{2}}\neq0}\mathbb E\exp\left\{\frac{\eta\left|\int^{1}_{0}f(r)\theta_{\epsilon}(r)dr\right|^2}{\int_0^1f^2(r)dr}\right\}
<\infty.
\end{equation*}
\end{lem}


\begin{lem}\label{lem-Y-Ym}
For any ~$\delta>0$, we have
$$
\limsup_{m\to\infty}\limsup_{\epsilon\to0}\epsilon^{2}
\log\mathbb P\Big(\epsilon\left|G_{m,\epsilon}(t_1,\cdots, t_k)-G_{\epsilon}(t_1,\cdots, t_k)\right|>\delta\Big)
=-\infty.
$$
\end{lem}

\begin{proof}
Without loss of generality, we may assume that there exists some $1\leq i\leq k$, $\int_0^1\big|K(t_i,r)-K_{m}(t_i,r)\big|^2dr\neq0$.
Otherwise, we have $\int^{1}_{0}K(t_i,r)\theta_{\epsilon}(r)dr=\int^{1}_{0}K_{m}(t_{i},r)\theta_{\epsilon}(r)dr$ for all $1\leq i\leq k$, which corresponds to the trivial case.

For any ~$0<\eta<\frac{1}{2}$,  it follows that
\begin{equation*}
\begin{aligned}
&\mathbb P\Big(\epsilon\left|G_{m,\epsilon}(t_1,\cdots, t_k)-G_{\epsilon}(t_1,\cdots, t_k)\right|>\delta\Big)\\
&\leq k\max_{1\leq i\leq k}\mathbb P\left(\left|\int^{1}_{0}K(t_i,r)\theta_{\epsilon}(r)dr-\int^{1}_{0}K_{m}(t_{i},r)\theta_{\epsilon}(r)dr\right|
>(k\epsilon)^{-1}\delta\right)\\
&\leq k\max_{i:\int_0^1\big|K(t_i,r)-K_{m}(t_i,r)\big|^2dr\neq0}\mathbb E\exp\left\{\frac{\eta \Big|\int^{1}_{0}(K(t_{i},r)-K_{m}(t_{i},r))\theta_{\epsilon}(r)dr\Big|^{2}}{\int_0^1\big|K(t_i,r)-K_{m}(t_i,r)\big|^2dr}\right\}\\
&\quad\cdot\exp\left\{-\frac{\eta k^{-2}\epsilon^{-2}\delta^{2}}{\max_{1\leq i\leq k}\int_0^1\big|K(t_i,r)-K_{m}(t_i,r)\big|^2dr}\right\}\\
&\leq C\exp\left\{-\frac{\eta k^{-2}\epsilon^{-2}\delta^{2}}{\max_{1\leq i\leq k}\int_0^1\big|K(t_i,r)-K_{m}(t_i,r)\big|^2dr}\right\},
\end{aligned}
\end{equation*}
where the last inequality is obtained by Lemma~\ref{lem-exponential estimation}, and $C>0$ is a uniform constant.
Consequently, we can get
\begin{align*}
&\limsup_{m\to\infty}\limsup_{\epsilon\to0}\epsilon^{2}
\log\mathbb P\Big(\epsilon\left|G_{m,\epsilon}(t_1,\cdots, t_k)-G_{\epsilon}(t_1,\cdots, t_k)\right|>\delta\Big)=-\infty,
\end{align*}
which completes the proof of this lemma.
\end{proof}

By Theorem 4.2.16 in Dembo and Zeitouni~(\cite{DZ-1998}) and Lemma~\ref{lem-Y-Ym},
to get large deviations for~$\epsilon G_{\epsilon}(t_1,\cdots, t_k)$, it suffices to study that of
~$\epsilon G_{m,\epsilon}(t_1,\cdots, t_k)$.

\begin{lem}\label{lem-finite dimen MDP}
The family~$\Big\{\epsilon G_{m, \epsilon}(t_1,\cdots, t_k), ~\epsilon>0\Big\}$ obeys LDP with speed ~$\epsilon^2$ and  rate function
\begin{align*}
&I_{m, t_{1},\cdots, t_{k}}(x_{1}, \cdots, x_{k})\\
&=\inf\left\{J(\varphi); ~\varphi(t)=\int^{t}_{0}\phi(r)dr\in\mathcal{AC},  \int^{1}_{0}K_{m}(t_{i},r)\phi(r)dr=x_{i}, 1\leq i\leq k, t\in[0,1]\right\},
\end{align*}
where $J(\varphi)$ is defined by \eqref{rate function-main result}.
\end{lem}

\begin{proof}
By (\ref{def-Ym-epsilon}), we write
\begin{align*}
G_{m,\epsilon}(t)&=\int_0^1K_m(t,r)d\Theta_{\epsilon}(r)\\
&=\Theta_{\epsilon}(1)K_m(t,1)-\int^{1}_{0} \Theta_{\epsilon}(r)\frac{\partial K_m(t,r)}{\partial r}dr,
\end{align*}
from which we see that $G_{m,\epsilon}(t)$ is a continuous functional of  $\Theta_{\epsilon}$ for fixed~$t\in[0,1]$. By contraction principle and a direct
calculation, we finish the proof.
\end{proof}

Now, we end this subsection by the large deviations of ~$\Big\{\epsilon G_{\epsilon}(t_1,\cdots, t_k), ~\epsilon>0\Big\}$.

\begin{pro}\label{pro-finite dimen MDP}
The family~$\Big\{\epsilon G_{\epsilon}(t_1,\cdots, t_k), ~\epsilon>0\Big\}$ satisfies the weak large deviations
with speed~$\epsilon^{2}$ and rate function
\begin{equation*}
\begin{aligned}
&I_{t_{1}, \cdots, t_{k}}(x_1, x_2,\cdots, x_k)\\
&=\inf\left\{J(\varphi); ~\varphi(t)=\int^{t}_{0}\phi(r)dr\in\mathcal{AC},~\int^{1}_{0}K(t_{i}, r)\phi(r)dr=x_{i}, 1\leq i\leq k,~t\in[0,1]\right\},
\end{aligned}
\end{equation*}
where~$J(\varphi)$ is defined by \eqref{rate function-main result}.
\end{pro}

\begin{proof}
By Theorem~4.2.16 in Dembo and Zeitouni~(\cite{DZ-1998}) and Lemma~\ref{lem-finite dimen MDP},
The family
$$
\Big\{\epsilon G_{\epsilon}(t_{1},\cdots, t_{k}), \epsilon>0\Big\}
$$
satisfies the weak large deviations
with speed~$\epsilon^{2}$ and rate function
\begin{equation}\label{finite dimensional rate function}
\widetilde{I}_{t_{1},\cdots, t_{k}}(x_{1},\cdots, x_{k})
=\sup_{\delta>0}\liminf_{m\to\infty}\inf_{|y-x| \le \delta}I_{m, t_{1}, \cdots, t_{k}}(y_{1},\cdots, y_{k}), ~~x=(x_{1}, \cdots,x_{k}).
\end{equation}
By the definition of $\Lambda^{*}$, it is noted that $|\varphi'(t)|\leq1$, a.e. $t\in[0,1]$ in the finite domain of $J(\cdot)$ and
$$
1-\sqrt{1-x^{2}}=\frac{x^{2}}{1+\sqrt{1-x^{2}}}, ~|x|\leq1,
$$
the classical compact argument yields
$$
I_{t_{1},\cdots, t_{k}}(x)=\widetilde{I}_{t_{1},\cdots, t_{k}}(x), ~~x\in\mathbb R^{k}.
$$
\end{proof}

\subsubsection{Exponential tightness and proof of Theorem~\ref{thm-main-result}}

For the exponential tightness of $\{\epsilon G_{\epsilon},\epsilon>0\}$,
we utilize Besov-L\'evy modulus embedding method illustrated by Theorem A.19 in Friz and Victoir ~(\cite{FV-2010}).

\begin{pro}\label{pro-expansion tight}
The family~$\{\epsilon G_{\epsilon},\epsilon>0\}$ is exponentially tight in $C([0,1])$.
\end{pro}

\begin{proof}
Since ~$\mathcal{G}$ is strictly increasing and continuous, then its inverse function $\mathcal{G}^{-1}$ is also strictly increasing and continuous.
Define
$
\mathbb{G}_{\epsilon}(t)=G_{\epsilon}\big(\mathcal{G}^{-1}(t)\big),~ t\in[0,\mathcal{G}(1)].
$
Note that
$$
G_{\epsilon}\big(\mathcal{G}^{-1}(t)\big)-G_{\epsilon}\big(\mathcal{G}^{-1}(s)\big)
=\int^{1}_{0}(K(\mathcal{G}^{-1}(t), r)-K(\mathcal{G}^{-1}(s), r))\theta_{\epsilon}(r)dr.
$$
Without  loss of generality, we may assume that $\int^{1}_{0}(K(\mathcal{G}^{-1}(t), r)-K(\mathcal{G}^{-1}(s), r))^{2}dr\neq0$,
otherwise we have $G_{\epsilon}\big(\mathcal{G}^{-1}(t)\big)-G_{\epsilon}\big(\mathcal{G}^{-1}(s)\big)=0$,  which corresponds to the trivial case.

Moreover, for $t\neq s$, $d(t,s):=|t-s|^{\alpha/2}$, where $\alpha$ is given in {\bf (H2)}. Then we have
\begin{align*}
&\frac{|\mathbb{G}_{\epsilon}(t)-\mathbb{G}_{\epsilon}(s)|^{2}}{d^2(t,s)}\\
&=\frac{|\mathbb{G}_{\epsilon}(t)-\mathbb{G}_{\epsilon}(s)|^{2}}{\int^{1}_{0}(K(\mathcal{G}^{-1}(t), r)-K(\mathcal{G}^{-1}(s), r))^{2}dr}
\cdot\frac{\int^{1}_{0}(K(\mathcal{G}^{-1}(t), r)-K(\mathcal{G}^{-1}(s), r))^{2}dr}{d^2(t,s)}.
\end{align*}
Therefore, Lemma~\ref{lem-exponential estimation} and assumption~(H2) imply that
$$
\sup_{t,s\in[0,1]}\mathbb E\exp\left\{\frac{\eta\|\mathbb{G}_{\epsilon}(t)-\mathbb{G}_{\epsilon}(s)|^{2}}{|t-s|^{\alpha}}\right\}
<\infty,
$$
where $0<\eta<\frac{1}{2}$.
Using Theorem A.19 in Friz and Victoir~(\cite{FV-2010}), there exists $c=c(\alpha, \alpha')$ such that
$$
\mathbb E\exp\left\{c\eta\sup_{t,s\in[0,1]}\frac{|\mathbb{G}_{\epsilon}(t)-\mathbb{G}_{\epsilon}(s)|^{2}}{|t-s|^{\alpha'}}\right\}
<\infty,
$$
where~$0<\alpha'<\alpha$. Hence, for any~$a>0$, it follows that
\begin{align*}
&\limsup_{\delta\to0}\limsup_{\epsilon\to0}\epsilon^{2}\log\mathbb P\left(\sup_{|t-s|<\delta}|\epsilon\mathbb{G}_{\epsilon}(t)-\epsilon\mathbb{G}_{\epsilon}(s)|>a\right)\\
&\leq \limsup_{\delta\to0}\limsup_{\epsilon\to0}\epsilon^{2}
\left(-\frac{c\eta a^2}{\epsilon^2\delta^{\alpha'}}
+\log\mathbb E\exp\left\{c\eta\sup_{t,s\in[0,1]}\frac{|\mathbb{G}_{\epsilon}(t)-\mathbb{G}_{\epsilon}(s)|^{2}}{|t-s|^{\alpha'}}\right\}\right)\\
&=-\infty,
\end{align*}
which implies the exponential tightness of ~$\{\epsilon\mathbb{G}_{\epsilon},\epsilon>0\}$, i.e.,
for any $L>0$, there exists a compact set $\mathcal K_{L}\subset C([0,\mathcal{G}(1)])$ such that
$$
\limsup_{\epsilon\to0}\epsilon^{2}\log\mathbb P\Big(\epsilon\mathbb{G}_{\epsilon}\in\mathcal K^{c}_{L}\Big)
\leq-L.
$$

Since $\mathcal{G}$ is strictly increasing and continuous,
$
\tilde{\mathcal K}_{L}:=\Big\{f: f=g\circ \mathcal{G}~\text{for ~some}~ g\in\mathcal K_{L}\Big\}
$
is also a compact set in $C([0,1])$ and
$f\in\tilde{\mathcal K}_{L}$ if and only if $f\circ \mathcal{G}^{-1}\in\mathcal K_{L}$.
Therefore,
$$
\limsup_{\epsilon\to0}\epsilon^{2}\log\mathbb P\Big(\epsilon G_{\epsilon}\in\tilde{\mathcal K}^{c}_{L}\Big)
\leq
\limsup_{\epsilon\to0}\epsilon^{2}\log\mathbb P\Big(\epsilon\mathbb{G}_{\epsilon}\in\mathcal K^{c}_{L}\Big)
\leq-L,
$$
which implies the desired result.
\end{proof}

\subsubsection{Proof of Theorem~\ref{thm-main-result}}
\begin{proof}[\noindent\textbf{\emph{Proof of Theorem~\ref{thm-main-result}}}]
By Proposition~\ref{pro-finite dimen MDP} and Proposition~\ref{pro-expansion tight},
the family~$\{\epsilon G_{\epsilon},\epsilon>0\}$ satisfies the functional
large deviation principle in $C([0,1])$, with speed ~${\epsilon}^{2}$ and rate function ~$\widetilde{I}(\cdot)$ defined by
$$
\widetilde{I}(f)=\sup_{k\in\mathbb{N}}\sup_{0\leq t_1\leq\cdots\leq t_k}I_{t_1,\cdots, t_k}\big(f(t_1),\cdots,f(t_k)\big), ~f\in C([0,1]),
$$
where~$I_{t_1,\cdots, t_k}(\cdot)$ is defined in Proposition~\ref{pro-finite dimen MDP}.

To end the proof of this theorem, we only need to show ~$\widetilde{I}(f)=I(f)$ for~$f\in C([0,1])$,
where~$I(\cdot)$ is defined by~(\ref{thm-main-result}).
In fact, by the definition of $\Lambda^{*}$, it is noted that $|\varphi'(t)|\leq1$, a.e. $t\in[0,1]$ in the finite domain of $J(\cdot)$.
Therefore, it can easily be seen that the classical compact argument yields $\widetilde{I}(f)=I(f)$. This ends the proof.
\end{proof}

\section{Proof of Theorem~\ref{thm-multiplicative noise}} \label{s:4}

Recall
\begin{align}  \label{e:SDE-LDP}
X^{\epsilon}(t)
&=x_{0}+\int^{t}_{0}b(X^{\epsilon}(s))ds+\lambda(\epsilon)\int^{t}_{0}\sigma(X^{\epsilon}(s))\theta_{\epsilon}(s)ds,
\end{align}
where~$\theta_{\epsilon}(t)=\frac{1}{\epsilon}(-1)^{N^{\epsilon^{-2}}(t)}$, $N^{\epsilon^{-2}}(t)=\int^{t}_{0} N^{\epsilon^{-2}}(ds)$, $N^{\alpha}$ is a Poisson random measure with the intensity measure $\alpha dt$ on ~$(\mathbb R_{+}, \mathcal B(\mathbb R_{+}))$ and $\tilde{N}^{\alpha}(ds)={N}^{\alpha}(ds)-\alpha ds$ is the compensated Poisson random measure.

By It\^o's formula (\cite{Applebaum}), we have
\begin{equation*}
\begin{aligned}
\theta_{\epsilon}(t)
&=\epsilon^{-1}-2\int^t_0 \theta_{\epsilon}(s-)\widetilde{N}^{\epsilon^{-2}}(ds)
-2\epsilon^{-2}\Theta_{\epsilon}(t),
\end{aligned}
\end{equation*}
i.e.,
\begin{equation}\label{Theta-theta-relation}
\Theta_{\epsilon}(t)
=\frac{\epsilon}{2}-\epsilon^{2}\int^t_0 \theta_{\epsilon}(s-)\widetilde{N}^{\epsilon^{-2}}(ds)
-\frac{\epsilon^{2}}{2}\theta_{\epsilon}(t).
\end{equation}
Therefore, \eqref{e:SDE-LDP} can be rewritten as
\begin{equation}\label{represe-sde}
\begin{aligned}
X^{\epsilon}(t)
&=x_{0}+\int^{t}_{0}b(X^{\epsilon}(s))ds
-\lambda(\epsilon)\epsilon^{2}\int^{t}_{0}\sigma(X^{\epsilon}(s-))\theta_{\epsilon}(s-)\tilde{N}^{\epsilon^{-2}}(ds)\\
&\ \ \ \ -\frac{\lambda(\epsilon)\epsilon^{2}}{2}\int_0^t\sigma(X^{\epsilon}(s-))d\theta_{\epsilon}(s).
\end{aligned}
\end{equation}
As we shall see below, the second and the third terms on the r.h.s. of \eqref{represe-sde} are drift and martingale respectively, which can be handled by standard method, while the last one need to be carefully analyzed.

\subsection{Exponential tightness of ~$\{X^{\epsilon}(t), t\in[0,1]\}$} It is clear that the exponential tightness immediately follows from the following two lemmas.

\begin{lem}\label{est to uniform bound}
Under the condition {\bf (H3)}, we have
\begin{equation}\label{eq-est to uniform bound}
\lim_{K\to\infty}\limsup_{\epsilon\to0}\lambda^2(\epsilon)\log\mathbb P\Big(\sup_{t\in[0,1]}|X^{\epsilon}(t)|>K\Big)=-\infty.
\end{equation}
\end{lem}
\begin{proof}
Denote $\tau_1,\tau_2,...$ the jump times of $\{N^{\epsilon^{-2}}(s)\}_{0 \le s \le t}$ for $t \in [0,1]$ and $\tau_{0}=0$, we have
\begin{equation*}
\begin{split}
\left|\int_0^t\sigma(X^{\epsilon}(s-))d\theta_{\epsilon}(s)\right|&=\left|\sum_{j \ge 1} \sigma(X^{\epsilon}(\tau_j-)) \left[\frac{1}{\e}(-1)^{N^{\epsilon^{-2}}(\tau_j-)+1}-\frac{1}{\e} (-1)^{N^{\epsilon^{-2}}(\tau_j-)}\right] \right|  \\
&=\frac{2}{\e} \left|\sum_{j \ge 1} \sigma(X^{\epsilon}(\tau_j-)) (-1)^{N^{\epsilon^{-2}}(\tau_j-)}\right|  \\
&\le \frac{2}{\e} \sum_{j \ge 1} \left|\sigma(X^{\epsilon}(\tau_j-))-\sigma(X^{\epsilon}(\tau_{j-1}-))\right|+\frac{2}{\epsilon}|\sigma(x_{0})| \\
& \le \frac{2L}{\e}  \sum_{j \ge 1}  \left|X^{\epsilon}(\tau_j-)-X^{\epsilon}(\tau_{j-1}-)\right|+\frac{2}{\epsilon}|\sigma(x_{0})|,
\end{split}
\end{equation*}
where the first inequality is because $(-1)^{N^{\epsilon^{-2}}(\tau_j-)}$ and $(-1)^{N^{\epsilon^{-2}}(\tau_{j-1}-)}$ has different signs for all
$j$, and the second one is by Lipschitz condition of $\sigma$. Since $\tau_j \in [0,t]$ and $\tau_{j-1}\le \tau_j $ for all $j$, we know by (\ref{e:SDE-LDP})
\begin{equation*}
\begin{split}
\sum_{j \ge 1}  \left|X^{\epsilon}(\tau_j-)-X^{\epsilon}(\tau_{j-1}-)\right| & \le \int_0^t |b(X^\e(s))| ds+\lambda(\epsilon) \int_0^t |\sigma(X^\e(s))| |\theta_\e(s)|ds \\
& \le \int_0^t |b(X^\e(s))| ds+\frac{\lambda(\epsilon)}{\e} \int_0^t |\sigma(X^\e(s))| ds.
\end{split}
\end{equation*}
Hence,
 \begin{equation}  \label{e:Tight-Bound}
\begin{aligned}
&\left|\int_0^t\sigma(X^{\epsilon}(s-))d\theta_{\epsilon}(s)\right|\\
& \le \frac{2L}{\e} \left(\int_0^t |b(X^\e(s))| ds+\frac{\lambda(\epsilon)}{\e} \int_0^t |\sigma(X^\e(s))| ds+\frac{|\sigma(x_{0})|}{L}\right).
\end{aligned}
\end{equation}
Combining \eqref{represe-sde} and \eqref{e:Tight-Bound}, we have
$$
\check{Y}^{\epsilon}(t)\leq X^{\epsilon}(t)\leq\hat{Y}^{\epsilon}(t),  \ \ \ \ \ \forall \  \ t \in [0,1],
$$
with
\begin{equation*}
\begin{aligned}
\hat{Y}^{\epsilon}(t)
&=x_{0}+\vartheta\int^{t}_{0}(1+|X^{\epsilon}(s)|)ds
-\lambda(\epsilon)\epsilon\int^{t}_{0} \sigma(X^{\epsilon}(s-))\epsilon\theta_{\epsilon}(s-)\tilde{N}^{\epsilon^{-2}}(ds)\\
\end{aligned}
\end{equation*}
\begin{equation*}
\begin{aligned}
\check{Y}^{\epsilon}(t)
&=x_{0}-\vartheta\int^{t}_{0}(1+|X^{\epsilon}(s)|)ds
-\lambda(\epsilon)\epsilon\int^{t}_{0}\sigma(X^{\epsilon}(s-))\epsilon\theta_{\epsilon}(s-)\tilde{N}^{\epsilon^{-2}}(ds).
\end{aligned}
\end{equation*}
where~$\vartheta$ is a sufficiently large and positive constant depending only on $L$, $b(0)$, $\sigma(0)$ and $\sigma(x_{0})$.

Applying Proposition~\ref{comparison est lem} below with $\rho=1$ and $B=M=\sqrt{2}\vartheta$, we get
\begin{align*}
&\lambda^2(\epsilon)\log\mathbb P\Big(\sup_{t\in[0,1]}|X^{\epsilon}(t)|>K\Big)\\
&\leq
4\vartheta+8\vartheta^2\epsilon^{2}(2+3\lambda^{2}(\epsilon))e^{4\sqrt{2}\vartheta\epsilon/{\lambda(\epsilon)}}
+\log\Big(\frac{1+2|x_{0}|^{2}}{1+K^{2}}\Big),
\end{align*}
which implies
\begin{equation*}
\begin{aligned}
&\lim_{K\to\infty}\limsup_{\epsilon\to0}\lambda^2(\epsilon)\log\mathbb P\Big(\sup_{t\in[0,1]}|X^{\epsilon}(t)|>K\Big)=-\infty.
\end{aligned}
\end{equation*}
\end{proof}
\begin{lem}\label{pro-exponential tightness-sde}
Under the condition {\bf(H3)} and any $\delta_0 \in (0,1/4)$, we have
\begin{equation}\label{eq-pro-exponential tightness-sde}
\lim_{\delta\to0}\limsup_{\epsilon\to0}\lambda^2(\epsilon)
\log\mathbb P\Big(\sup_{0\leq t-s\leq\delta, t,s\in[0,1]}\big|X^{\epsilon}(t)-X^{\epsilon}(s)\big|>\delta_0\Big)=-\infty.
\end{equation}
\end{lem}

\begin{proof}
Denote
$$A^{\e,\delta,\delta_0}=\Big\{\sup_{0\leq t-s\leq\delta, t,s\in[0,1]}\big|X^{\epsilon}(t)-X^{\epsilon}(s)\big|>\delta_0\Big\}$$
and
$$\tilde A^{\e,\delta,\delta_0}=
\Big\{\sup_{k\leq[\frac{1}{\delta}]}\sup_{t\in[s_{k},(s_{k}+\delta)\wedge1]}\big|X^{\epsilon}(t)-X^{\epsilon}(s_{k})\big|>\frac{\delta_0}{3}\Big\},$$
where $\{s_{k}\}$ is a sequence of $\delta$-partition of $[0,1]$. By the following easy fact:
\begin{equation*}
\sup_{0\leq t-s\leq\delta, t,s\in[0,1]}\big|X^{\epsilon}(t)-X^{\epsilon}(s)\big|\leq3\sup_{k\leq[\frac{1}{\delta}]}\sup_{t\in[s_{k},(s_{k}+\delta)\wedge1]}|X^{\epsilon}(t)-X^{\epsilon}(s_{k})\big|,
\end{equation*}
we know
\begin{equation*} \label{e:Ine-0}
\mathbb P(A^{\e,\delta,\delta_0}) \le \mathbb P(\tilde A^{\epsilon,\delta,\delta_0}).
\end{equation*}
It suffices to prove that the limit holds for $\mathbb P(\tilde A^{\epsilon,\delta,\delta_0})$.

Further denote
$$B^{\e,K}=\Big\{\sup_{t\in[0,1]}|X^{\epsilon}(t)|<K\Big\} \ \ \ \ {\rm for} \ K>0,$$
$$B^{\e,\delta,\delta_0}=\Big\{\sup_{k\leq[\frac{1}{\delta}]}\sup_{t\in[s_{k},(s_{k}+\delta)\wedge1]}\lambda(\epsilon)\epsilon^2 \left|\int^{t}_{s_{k}}\sigma(X^{\epsilon}(s-))
\theta_{\epsilon}(s-)\tilde{N}^{\epsilon^{-2}}(ds)\right|<\delta^{2}_{0}\Big\}.$$
It is easy to see that
\begin{equation}
\mathbb P(\tilde A^{\epsilon,\delta,\delta_0})
\le \mathbb P(\tilde A^{\epsilon,\delta,\delta_0}\cap B^{\e,K} \cap B^{\e,\delta,\delta_0})
+\mathbb P((B^{\e,\delta,\delta_0})^c\cap B^{\e,K})+\mathbb P((B^{\e,K})^c).
\end{equation}
We shall show that as $\delta$ and $\epsilon$ are sufficiently small
\begin{equation}  \label{e:Ine-1}
 \tilde A^{\epsilon,\delta,\delta_0}\cap B^{\e,K} \cap B^{\e,\delta,\delta_0}=\emptyset,
\end{equation}
and that
\begin{equation}  \label{e:Ine-2}
\lim_{\delta\to0}\limsup_{\epsilon\to0}\lambda^2(\epsilon)
\log\mathbb P((B^{\e,\delta,\delta_0})^c\cap B^{\e,K})
=-\infty,
\end{equation}
\begin{equation}  \label{e:Ine-3}
\lim_{K\to\infty}\limsup_{\epsilon\to0}\lambda^2(\epsilon)\log\mathbb P((B^{\e,K})^c)=-\infty.
\end{equation}
Combining the previous relations, the limit immediately holds for $\mathbb P(\tilde A^{\epsilon,\delta,\delta_0})$ and we conclude the proof.

It remains to prove the three relations \eqref{e:Ine-1}, \eqref{e:Ine-2} and \eqref{e:Ine-3}, in which the last one has been proved in previous Lemma \ref{est to uniform bound}.

Let us first show \eqref{e:Ine-1}. For any $t\in[s_k, (s_k+\delta)\wedge1]$, on the event $ \tilde A^{\epsilon,\delta,\delta_0}\cap B^{\e,K} \cap B^{\e,\delta,\delta_0}$, by \eqref{represe-sde} and \eqref{e:Tight-Bound} we have
\begin{equation*}
\begin{aligned}
&|X^{\epsilon}(t)-X^{\epsilon}(s_k)|\\
 & \le \int^{t}_{s_k}|b(X^{\epsilon}(s))|ds
+\lambda(\epsilon)\epsilon^{2}\left|\int^{t}_{s_k}\sigma(X^{\epsilon}(s-))\theta_{\epsilon}(s-)\tilde{N}^{\epsilon^{-2}}(ds)\right|+\frac{\lambda(\epsilon)\epsilon^{2}}{2}\left|\int^{t}_{s_{k}}\sigma(X^{\epsilon}(s-))d\theta_{\epsilon}(s)\right| \\
&\le \int^{t}_{s_{k}}|b(X^{\epsilon}(s))|ds+\delta^{2}_{0}+L\lambda(\epsilon)\epsilon \int^{t}_{s_{k}}|b(X^{\epsilon}(s))|ds+L\lambda^2(\epsilon)\int^{t}_{s_{k}}|\sigma(X^{\epsilon}(s))|ds+\lambda(\epsilon)\epsilon|\sigma(x_{0})|\\
&\leq(1+L\lambda(\epsilon)\epsilon)(KL+|b(0)|)\delta+\delta^{2}_{0}+L\lambda^2(\epsilon)(KL+|\sigma(0)|)\delta+\lambda(\epsilon)\epsilon|\sigma(x_{0})|,
\end{aligned}
\end{equation*}
where the third inequality is obtained by Lipschitz condition of $b$ and $\sigma$ and $\sup_{0 \le t \le 1} |X^\e(t)|<K$.

Hence, as $\delta$ and $\epsilon$ are sufficiently small and $\delta_{0}<\frac{1}{4}$, from the above bound we know on the event $\tilde A^{\epsilon,\delta,\delta_0}\cap B^{\e,K} \cap B^{\e,\delta,\delta_0}$
\begin{equation*}
\sup_{t \in [s_k, (s_k+\delta)\wedge1]}|X^{\epsilon}(t)-X^{\epsilon}(s_k)| \le \frac{\delta_0}4.
\end{equation*}
Because $k$ is arbitrary, we know on the event $\tilde A^{\epsilon,\delta,\delta_0}\cap B^{\e,K} \cap B^{\e,\delta,\delta_0}$
\begin{equation*}
\sup_{k \le [1/\delta]}\sup_{t \in [s_k, (s_k+\delta)\wedge1]}|X^{\epsilon}(t)-X^{\epsilon}(s_k)| \le \frac{\delta_0}4,
\end{equation*}
which is in the set $(\tilde{A}^{\e,\delta,\delta_0})^c$. Hence, \eqref{e:Ine-1} holds.

Now let us prove \eqref{e:Ine-2}. To proceed the proof, for fixed $k\leq [1/\delta]$, we define the following stopping time:
\begin{align*}
&\xi_{k}=\inf\Big\{t\geq s_{k}; |X^{\epsilon}(t)|\geq K\Big\}.
\end{align*}
It is easy to see that
\begin{equation}\label{delta expansion}
\begin{aligned}
&P\big((B^{\e,\delta,\delta_0})^c\cap B^{\e,K}\big)\\
&\leq\sum_{k=0}^{[1/\delta]} \mathbb P\Big(\sup_{t\in[s_{k},(s_{k}+\delta)\wedge1]}\lambda(\epsilon)\epsilon^2
\left|\int^{t}_{s_{k}}\sigma(X^{\epsilon}(s-))\theta_{\epsilon}(s-)\tilde{N}^{\epsilon^{-2}}(ds)\right|\geq\delta^{2}_{0}, B^{\e,K}\Big)\\
&\leq\sum_{k=0}^{[1/\delta]} \mathbb P\Big(\sup_{t\in[s_{k},(s_{k}+\delta)\wedge1]}\lambda(\epsilon)\epsilon^2
\left|\int^{t}_{s_{k}}\sigma(X^{\epsilon}(s-))\theta_{\epsilon}(s-)\tilde{N}^{\epsilon^{-2}}(ds)\right|\geq\delta^{2}_{0}, \xi_{k}\geq1\Big)\\
&\leq\sum_{k=0}^{[1/\delta]}
\mathbb P\Big(\sup_{t\in[s_{k},(s_{k}+\delta)\wedge\xi_{k}\wedge1]}\lambda(\epsilon)\epsilon^2
\left|\int^{t}_{s_{k}}\sigma(X^{\epsilon}(s-))\theta_{\epsilon}(s-)\tilde{N}^{\epsilon^{-2}}(ds)\right|\geq\delta^{2}_{0}\Big)
\end{aligned}
\end{equation}
Using Proposition \ref{exponential inequality} with
$$
\rho=LK+|\sigma(0)|, 0\leq\iota_2-\iota_1\leq\delta, a=\delta^{2}_{0},
$$
we get for any $\eta>0$
\begin{equation}\label{est to B123-epsilon}
\begin{aligned}
&\limsup_{\epsilon\to0}\lambda^2(\epsilon)
\log\max_{k\leq[1/\delta]}\mathbb P\Big(\sup_{t\in[s_{k},(s_{k}+\delta)\wedge\xi_{k}\wedge1]}\lambda(\epsilon)\epsilon^2
\left|\int^{t}_{s_{k}}\sigma(X^{\epsilon}(s-))\theta_{\epsilon}(s-)\tilde{N}^{\epsilon^{-2}}(ds)\right|\geq\delta^{2}_{0}\Big)\\
& \leq-\eta\delta_0^2+\frac{1}{2}\eta^2\big(LK+|\sigma(0)|\big)^2\delta\limsup_{\epsilon\to0} e^{\epsilon\lambda^{-1}(\epsilon)\eta\big(LK+|\sigma(0)|\big)}\\
&=\left\{\begin{array}{ll}
-\eta\delta_0^2+\frac{1}{2}\eta^2\big(LK+|\sigma(0)|\big)^2\delta, & \lim_{\epsilon\to0}\frac{\lambda(\epsilon)}{\epsilon}=+\infty;\\
-\eta\delta_0^2+\frac{1}{2}\eta^2\big(LK+|\sigma(0)|\big)^2\delta e^{\kappa^{-1}\eta\big(LK+|\sigma(0)|\big)},
& \lim_{\epsilon\to0}\frac{\lambda(\epsilon)}{\epsilon}=\kappa.\\\end{array}\right.
\end{aligned}
\end{equation}
By \eqref{delta expansion}, it can easily be seen that
\begin{align*}
&\limsup_{\epsilon\to0}\lambda^2(\epsilon)
\log\mathbb P((B^{\e,\delta,\delta_0})^c\cap B^{\e,K})\\
&=\limsup_{\epsilon\to0}\lambda^2(\epsilon)
\log\left([1/\delta]+1\right)\\
&\quad+\limsup_{\epsilon\to0}\lambda^2(\epsilon)
\log\max_{k\leq[1/\delta]}\mathbb P\Big(\sup_{t\in[s_{k},(s_{k}+\delta)\wedge\xi_{k}\wedge1]}\lambda(\epsilon)\epsilon^2
\left|\int^{t}_{s_{k}}\sigma(X^{\epsilon}(s-))\theta_{\epsilon}(s-)\tilde{N}^{\epsilon^{-2}}(ds)\right|\geq\delta^{2}_{0}\Big)\\
&=\limsup_{\epsilon\to0}\lambda^2(\epsilon)
\log\max_{k\leq[1/\delta]}\mathbb P\Big(\sup_{t\in[s_{k},(s_{k}+\delta)\wedge\xi_{k}\wedge1]}\lambda(\epsilon)\epsilon^2
\left|\int^{t}_{s_{k}}\sigma(X^{\epsilon}(s-))\theta_{\epsilon}(s-)\tilde{N}^{\epsilon^{-2}}(ds)\right|\geq\delta^{2}_{0}\Big).
\end{align*}
From \eqref{est to B123-epsilon}, the proof of \eqref{e:Ine-2} is completed  by letting $\delta\to0$ and then $\eta\to\infty$.
\end{proof}

\subsection{Proof of Theorem~\ref{thm-multiplicative noise}}
We will split the proof into two parts, i.e., lower bound and upper bound. \\

(1) {\bf{Lower bound.}} For any $\varphi \in C([0,1])$, we consider
\begin{align*}
\hat{X}^{\epsilon}(t)
&=x_{0}+\int^{t}_{0}b(\varphi(s))ds+\lambda(\epsilon)\int^{t}_{0}\sigma(\varphi(s))\theta_{\epsilon}(s)ds.
\end{align*}
By Theorem 2.1 in Jiang and Yang (\cite{JY-2021}) and Theorem~\ref{thm-main-result},
we obtain by contraction principle that~$\left\{\hat{X}^{\epsilon}(t), ~t\in[0,1]\right\}$
satisfies large deviations with speed ~$\lambda^2(\epsilon)$ and rate function  $\mathcal{I}^{\varphi}$, where
\begin{equation*}
\begin{aligned}
\mathcal{I}^{\varphi}(x)=&\inf\left\{\int^{1}_{0}\Gamma(\phi(r))dr; ~\phi\in L^2([0,1]) \ {\rm is \ such \ that \ }\right.\\
&\left.\quad x(t)=x_{0}+\int^{t}_{0}b(\varphi(s))ds+\int^{t}_{0}\sigma(\varphi(s))\phi(s)ds \right\}.
\end{aligned}
\end{equation*}
By the definition of $\mathcal{I}$ and $\mathcal{I}^{\varphi}$, it is easily checked that
\begin{equation}\label{embedding of rate function}
\mathcal{I}^{\varphi}(\varphi)=\mathcal{I}(\varphi).
\end{equation}
Observe by (\ref{represe-sde})
\begin{equation}\label{difference-X-tildeX}
\begin{aligned}
\hat{X}^{\epsilon}(t)-X^{\epsilon}(t)&=\int^{t}_{0}\big(b(\hat{X}^{\epsilon}(s))-b(X^{\epsilon}(s))\big)ds\\
&-\lambda(\epsilon)\epsilon\int^{t}_{0}\big(\sigma(\hat{X}^{\epsilon}(s-))-\sigma(X^{\epsilon}(s-))\big)\epsilon\theta^{\epsilon}(s-)\tilde{N}^{\epsilon^{-2}}(ds)\\
&+\mathcal{R}_{1,\epsilon}(t)+\mathcal{R}_{2,\epsilon}(t)+\mathcal{R}_{3,\epsilon}(t)+\mathcal{R}_{4,\epsilon}(t),
\end{aligned}
\end{equation}
where
\begin{align*}
&\mathcal{R}_{1,\epsilon}(t)=\int^{t}_{0}\big(b(\varphi(s))-b(\hat{X}^{\epsilon}(s)))\big)ds,\\
&\mathcal{R}_{2,\epsilon}(t)=\lambda(\epsilon)\int^{t}_{0}\big(\sigma(\varphi_{m}(s))-\sigma(\hat{X}^{\epsilon}(s))\big)\theta_{\epsilon}(s)ds,\\
&\mathcal{R}_{3,\epsilon}(t)=\lambda(\epsilon)\int^{t}_{0}\big(\sigma(\hat{X}^{\epsilon}(s))-\sigma(X^{\epsilon}(s))\big)\theta^{\epsilon}(s)ds\\
&\quad\quad\quad\quad
+\lambda(\epsilon)\epsilon\int^{t}_{0}\big(\sigma(\hat{X}^{\epsilon}(s-))-\sigma(X^{\epsilon}(s-))\big)\epsilon\theta^{\epsilon}(s-)\tilde{N}^{\epsilon^{-2}}(ds),\\
&\mathcal{R}_{4,\epsilon}(t)=\lambda(\epsilon)\int^{t}_{0}\big(\sigma(\varphi(s))-\sigma(\varphi_m(s))\big)\theta_{\epsilon}(s)ds,
\end{align*}
with $\{\varphi_m, m\geq1\}\in C^1([0,1])$ satisfying~$\lim_{m\to\infty}\sup_{t\in[0,1]}|\varphi(t)-\varphi_m(t)|=0$.

For~$\delta>0$, we define the following stopping times used later:
\begin{equation*}
\begin{aligned}
&\chi_i=\inf\Big\{t\geq0; |\mathcal{R}_{i,\epsilon}(t)|\geq \delta^{1/2}\Big\}, \quad i=1,2,3,\\
&\chi_4=\inf\Big\{t\geq0; |\mathcal{R}_{4,\epsilon}(t)|\geq \sup_{t\in[0,1]}|\varphi(t)-\varphi_m(t)|^{1/2}\Big\}
\end{aligned}
\end{equation*}
and further define
$$
\chi=\min_{1 \le i \le 4} \chi_i.
$$
Assume the following two relations hold (which will be proved later)
\begin{equation}\label{expon-ineq-X-tildeX}
\lim_{\delta\to0}\lim_{m\to\infty}\limsup_{\epsilon\to0}\lambda^2(\epsilon)
\log\mathbb P\Big(\sup_{t\in[0,\chi\wedge1]}|\hat{X}^{\epsilon}(t)-X^{\epsilon}(t)|\geq\delta^{1/4}\Big)
=-\infty,
\end{equation}
\begin{equation}\label{negligible-stoppingtime}
\lim_{\delta\to0}\lim_{m\to\infty}\limsup_{\epsilon\to0}\lambda^2(\epsilon)
\log\mathbb P\Big(\sup_{t\in[0,1]}|\hat{X}^{\epsilon}(t)-\varphi(t)|<\delta, \chi<1\Big)
=-\infty.
\end{equation}
Then, for any $a>0$, denote
$$
B(\varphi, a)=\Big\{x\in C([0,1],\mathbb R), \sup_{t\in[0,1]}|x(t)-\varphi(t)|<a\Big\}.
$$
For any $\gamma>0$, we can choose $\delta>0$ such that~$\delta+\delta^{1/4}<\gamma$.
Then, we have
\begin{align*}
&\Big\{\sup_{t\in[0,1]}|\hat{X}^{\epsilon}(t)-\varphi(t)|<\delta\Big\}\\
&=\Big(\Big\{\sup_{t\in[0,1]}|\hat{X}^{\epsilon}(t)-\varphi(t)|<\delta, \chi<1\Big\}\Big)\\
&\quad\cup\Big(\Big\{\sup_{t\in[0,1]}|\hat{X}^{\epsilon}(t)-\varphi(t)|<\delta, \chi \ge 1\Big\}
\cap\Big\{\sup_{t\in[0,1]}|\hat{X}^{\epsilon}(t)-{X}^{\epsilon}(t)|<\delta^{1/4}\Big\}\Big)\\
&\quad\cup\Big(\Big\{\sup_{t\in[0,1]}|\hat{X}^{\epsilon}(t)-\varphi(t)|<\delta, \chi \ge 1\Big\}
\cap\Big\{\sup_{t\in[0,1]}|\hat{X}^{\epsilon}(t)-{X}^{\epsilon}(t)|\geq\delta^{1/4}\Big\}\Big)\\
&\subset\Big(\Big\{\sup_{t\in[0,1]}|\hat{X}^{\epsilon}(t)-\varphi(t)|<\delta, \chi<1\Big\}\Big)\\
&\quad\cup\Big\{\sup_{t\in[0,1]}|{X}^{\epsilon}(t)-\varphi(t)|<\gamma\Big\}
\cup\Big\{\sup_{t\in[0,\chi\wedge1]}|\hat{X}^{\epsilon}(t)-X^{\epsilon}(t)|\geq\delta^{1/4}\Big\}.
\end{align*}
By (\ref{expon-ineq-X-tildeX}), \eqref{negligible-stoppingtime} and the following well known relation:
$$
\liminf_{\epsilon\to0}\lambda^2(\epsilon)\log(a(\epsilon)+b(\epsilon))
\leq \max\Big\{\liminf_{\epsilon\to0}\lambda^2(\epsilon)\log (a(\epsilon)),\limsup_{\epsilon\to0}\lambda^2(\epsilon)\log(b(\epsilon))\Big\},
$$
we immeidatley obtain
\begin{equation}\label{e:LamPB}
\begin{split}
\liminf_{\epsilon\to0}\lambda^2(\epsilon)\log\mathbb P\big(X^{\epsilon}\in B(\varphi,\gamma)\big)
& \geq \lim_{\delta\to0}\liminf_{\epsilon\to0}\lambda^2(\epsilon)\log\mathbb P\Big(\hat{X}^{\epsilon} \in B(\varphi,\delta)\Big)  \\
& \geq -\mathcal{I}^{\varphi}(\varphi)=-\mathcal{I}(\varphi),
\end{split}
\end{equation}
where the last equality is derived from \eqref{embedding of rate function}.

Letting $G \subset C([0,1])$ be an open set, for any $\varphi \in G$, there exists some $\eta>0$ so that $B(\varphi, \eta) \subset G$. By \eqref{e:LamPB}, we have
\begin{equation*}
\begin{split}
\liminf_{\epsilon\to0}\lambda^2(\epsilon)\log\mathbb P\big(X^{\epsilon}\in G\big) \ge \liminf_{\epsilon\to0}\lambda^2(\epsilon)\log\mathbb P\big(X^{\epsilon}\in B(\varphi,\eta)\big) \ge -\mathcal{I}(\varphi).
\end{split}
\end{equation*}
 Since $\varphi \in G$ is arbitrary, we know
 \begin{equation*}
\begin{split}
\liminf_{\epsilon\to0}\lambda^2(\epsilon)\log\mathbb P\big(X^{\epsilon}\in G\big)  \ge -\inf_{\varphi \in G}\mathcal{I}(\varphi).
\end{split}
\end{equation*}
So the lower bound of LDP is proved.


It remains to prove \eqref{expon-ineq-X-tildeX} and \eqref{negligible-stoppingtime}.
For $t\in[0,\chi\wedge1)$, from \eqref{difference-X-tildeX} it can easily be seen that
$$
\hat{Y}^{\epsilon}(t)\leq \hat{X}^{\epsilon}(t)-X^{\epsilon}(t)\leq \check{Y}^{\epsilon}(t),
$$
where
\begin{align*}
\check{Y}^{\epsilon}(t)
&=\big(3\delta^{1/2}+\sup_{t\in[0,1]}|\varphi(t)-\varphi_m(t)|^{1/2}\big)+\int^{t}_{0}\big(b(\hat{X}^{\epsilon}(s))-b(X^{\epsilon}(s))\big)ds\\
&\quad-\lambda(\epsilon)\epsilon\int^{t}_{0}\big(\sigma(\hat{X}^{\epsilon}(s-))-\sigma(X^{\epsilon}(s-))\big)\epsilon\theta^{\epsilon}(s-)\tilde{N}^{\epsilon^{-2}}(ds),
\end{align*}
and
\begin{align*}
\hat{Y}^{\epsilon}(t)
&=-\big(3\delta^{1/2}+\sup_{t\in[0,1]}|\varphi(t)-\varphi_m(t)|^{1/2}\big)+\int^{t}_{0}\big(b(\hat{X}^{\epsilon}(s))-b(X^{\epsilon}(s))\big)ds\\
&\quad-\lambda(\epsilon)\epsilon\int^{t}_{0}\big(\sigma(\hat{X}^{\epsilon}(s-))-\sigma(X^{\epsilon}(s-))\big)\epsilon\theta^{\epsilon}(s-)\tilde{N}^{\epsilon^{-2}}(ds).
\end{align*}
By Proposition~\ref{comparison est lem} with $\rho=0$ and $B=M=L$, we have for any $\delta>0$,
\begin{align*}
&\lambda^2(\epsilon)\log\mathbb P\Big(\sup_{t\in[0, \chi\wedge1]}|\hat{X}^{\epsilon}(t)-X^{\epsilon}(t)|\geq\delta^{1/4}\Big)\\
&\leq 2\sqrt{2}L+4L^2\epsilon^2(2+3\lambda^{2}(\epsilon))e^{4L\epsilon/{\lambda(\epsilon)}}
+\log\Big(\frac{2\big(3\delta^{1/2}+\sup_{t\in[0,1]}|\varphi(t)-\varphi_m(t)|^{1/2}\big)^{2}}{\delta^{1/2}}\Big),
\end{align*}
which immediately leads to \eqref{expon-ineq-X-tildeX}.

Now we are at the position to prove \eqref{negligible-stoppingtime}.  The proof will be divided into four parts according to the terms related to $\chi_{i}, i\leq 4$.

On the event $\Big\{\sup_{t\in[0,1]}|\hat{X}^{\epsilon}(t)-\varphi(t)|<\delta\Big\}$, we observe
$$
\sup_{t\in[0,1]}|\mathcal{R}_{1,\epsilon}(t)|\leq L\delta<\delta^{1/2}
$$
as $\delta$ is sufficiently small, this implies that the set $\{\sup_{t\in[0,1]}|\hat{X}^{\epsilon}(t)-\varphi(t)|<\delta, \chi_1<1\}$ is empty. Thus,
\begin{equation}\label{eq1-negligible-stoppingtime}
\lim_{\delta\to0}\lim_{\epsilon\to0}\lambda^2(\epsilon)
\log\mathbb P\Big(\sup_{t\in[0,1]}|\hat{X}^{\epsilon}(t)-\varphi(t)|<\delta, \chi_1<1\Big)
=-\infty.
\end{equation}

By (\ref{Theta-theta-relation}), we rewrite
\begin{equation*}
\begin{aligned}
\mathcal{R}_{2,\epsilon}(t)
&=-\lambda(\epsilon)\epsilon^{2}\int^{t}_{0}\big(\sigma(\varphi_{m}(s))-\sigma(\hat{X}^{\epsilon}(s-))\big)
\theta_{\epsilon}(s-)\tilde{N}^{\epsilon^{-2}}(ds)\\
&\quad-\frac{\lambda(\epsilon)\epsilon^{2}}{2}\int^{t}_{0}\big(\sigma(\varphi_{m}(s))-\sigma(\hat{X}^{\epsilon}(s))\big)d\theta_{\epsilon}(s),\\
\mathcal{R}_{3,\epsilon}(t)
&=-\frac{\lambda(\epsilon)\epsilon^{2}}{2}\int^{t}_{0}\big(\sigma(\hat{X}^{\epsilon}(s-))-\sigma(X^{\epsilon}(s-))\big)d\theta_{\epsilon}(s).
\end{aligned}
\end{equation*}
By the same argument for showing \eqref{e:Tight-Bound} and Lipschitz conditions for $b$ and $\sigma$, we obtain
\begin{equation}\label{estimate-mathcalR2}
\begin{aligned}
|\mathcal{R}_{2,\epsilon}(t)|
&\leq\vartheta\lambda^{2}(\epsilon)\int^{t}_{0}1+|\varphi(s)|+|\varphi'_m(s)|ds
\\
&+\lambda(\epsilon)\epsilon^{2}\Big|\int^{t}_{0}\big(\sigma(\varphi_{m}(s))-\sigma(\hat{X}^{\epsilon}(s-))\big)
\theta_{\epsilon}(s-)\tilde{N}^{\epsilon^{-2}}(ds)\Big|,
\end{aligned}
\end{equation}
and
\begin{equation}\label{estimate-mathcalR4}
\begin{aligned}
|\mathcal{R}_{3,\epsilon}(t)|
&\leq\vartheta\lambda^{2}(\epsilon)\int_0^t1+|\varphi(s)|+|{X}^{\epsilon}(s)|ds,
\end{aligned}
\end{equation}
where $\vartheta$ is a sufficiently large positive constant depending only on $L, b(x_0), b(0), \sigma(x_0)$, $\sigma(0)$ and $\sigma(\varphi_{m}(0))$.

For $\chi_2$, we define the following stopping time:
$$
\hat{\zeta}=\inf\Big\{t\geq0; |\hat{X}^{\epsilon}(t)-\varphi(t)|\geq\delta\Big\}.
$$
For any $\eta>0$, we observe  by \eqref{estimate-mathcalR2}
\begin{align*}
&\lim_{\epsilon\to0}\lambda^2(\epsilon)
\log\mathbb P\Big(\sup_{t\in[0,1]}|\hat{X}^{\epsilon}(t)-\varphi(t)|<\delta, \chi_2<1\Big)\\
&\leq\lim_{\epsilon\to0}\lambda^2(\epsilon)
\log\mathbb P\Big(\sup_{t\in[0,\hat{\zeta}\wedge1]}|\mathcal{R}_{2,\epsilon}(t)|\geq\delta^{1/2}\Big)\\
&\leq \lim_{\epsilon\to0}\lambda^2(\epsilon)
\log\mathbb P\Big(\sup_{t\in[0,\hat{\zeta}\wedge1]}
\lambda(\epsilon)\epsilon\Big|\int^{t}_{0}\big(\sigma(\varphi_{m}(s))-\sigma(\hat{X}^{\epsilon}(s-))\big)
\epsilon\theta_{\epsilon}(s-)\tilde{N}^{\epsilon^{-2}}(ds)\Big|\geq\frac{1}{2}\delta^{1/2}\Big)\\
&\leq -\frac{1}{2}\eta\delta^{1/2}+\frac{\eta^2L^{2}}{2}\big(\delta+\sup_{t\in[0,1]}|\varphi_m(t)-\varphi(t)|\big)^2
\lim_{\epsilon\to0}e^{\epsilon\lambda^{-1}(\epsilon)\eta L\big(\delta+\sup_{t\in[0,1]}|\varphi_m(t)-\varphi(t)|\big)},
\end{align*}
where the last inequality is obtained by the fact that for any $s\in[0,\hat{\zeta}\wedge1]$, $|\sigma(\varphi_{m}(s))-\sigma(\hat{X}^{\epsilon}(s-))|\leq L(\delta+\sup_{t\in[0,1]}|\varphi_m(t)-\varphi(t)|)$ and apply Proposition \ref{exponential inequality} with
$$
0\leq\iota_2-\iota_1\leq1,\quad \rho=L(\delta+\sup_{t\in[0,1]}|\varphi_m(t)-\varphi(t)|),\quad a=\frac{1}{2}\delta^{1/2}.
$$
Consequently, it holds
\begin{align*}
&\lim_{m\to\infty}\lim_{\epsilon\to0}\lambda^2(\epsilon)
\log\mathbb P\Big(\sup_{t\in[0,1]}|\hat{X}^{\epsilon}(t)-\varphi(t)|<\delta, \chi_2<1\Big)\\
&\leq \left\{\begin{array}{ll}
-\frac{1}{2}\eta\delta^{1/2}+\frac{\eta^2 L^{2}}{2}\delta^2, & \lim_{\epsilon\to0}\frac{\lambda(\epsilon)}{\epsilon}=+\infty;\\
-\frac{1}{2}\eta\delta^{1/2}+\frac{\eta^2L^{2}}{2}\delta^2
e^{\eta\delta L/{\kappa}}, & \lim_{\epsilon\to0}\frac{\lambda(\epsilon)}{\epsilon}=\kappa.\\\end{array}\right.
\end{align*}
Choosing $\eta=\delta^{-1}$, we have
\begin{equation}\label{eq2-negligible-stoppingtime-0}
\lim_{\delta\to0}\lim_{m\to\infty}\lim_{\epsilon\to0}\lambda^2(\epsilon)
\log\mathbb P\Big(\sup_{t\in[0,1]}|\hat{X}^{\epsilon}(t)-\varphi(t)|<\delta, \chi_2<1\Big)
=-\infty.
\end{equation}

For $\chi_3$,  as $\e$ is sufficiently small so that $\vartheta\lambda^{2}(\epsilon)\int_0^1(1+|\varphi(s)|+K)ds<\delta^{\frac{1}{2}}$,
by (\ref{estimate-mathcalR4}), $\Big\{\sup_{t\in[0,1]}|X^{\epsilon}(t)|<K,
\sup_{t\in[0,1]}|\mathcal{R}_{3,\epsilon}(t)|\geq\delta^{1/2}\Big\}$ is an empty set and thus
\begin{align*}
&\mathbb P\Big(\chi_3<1\Big)\\
&\leq\mathbb P\Big(\sup_{t\in[0,1]}|X^{\epsilon}(t)|<K,
\sup_{t\in[0,1]}|\mathcal{R}_{3,\epsilon}(t)|\geq\delta^{1/2}\Big)
+\mathbb{P}\Big(\sup_{t\in[0,1]}|X^{\epsilon}(t)|\geq K\Big)\\
&\leq\mathbb{P}\Big(\sup_{t\in[0,1]}|X^{\epsilon}(t)|\geq K\Big).
\end{align*}
By Lemma~\ref{est to uniform bound},
\begin{equation}\label{eq3-negligible-stoppingtime}
\begin{aligned}
&\lim_{\epsilon\to0}\lambda^2(\epsilon)
\log\mathbb P\Big(\chi_3<1\Big)
\leq\lim_{K\to+\infty}\lim_{\epsilon\to0}\lambda^2(\epsilon)
\log\mathbb P\Big(\sup_{t\in[0,1]}|X^{\epsilon}(t)|\geq K\Big)
=-\infty.
\end{aligned}
\end{equation}

For the term related to $\chi_4$,
recall that $\mathcal{R}_{4,\epsilon}(t)=\lambda(\epsilon)\int^{t}_{0}\big(\sigma(\varphi(u))-\sigma(\varphi_m(u))\big)\theta_{\epsilon}(u)du$,
we may assume that $\int^{t}_{s}|\sigma(\varphi(u))-\sigma(\varphi_m(u))|^{2}du\neq 0$ for $0\leq s<t\leq 1$ and $\sup_{u\in[0,1]}|\varphi(u)-\varphi_m(u)|^{2}\neq0$ without loss of generality.

We note that for $0\leq s<t\leq 1$
\begin{equation*}
\int^{t}_{s}|\sigma(\varphi(u))-\sigma(\varphi_m(u))|^{2}du\leq L^{2}(t-s)\sup_{u\in[0,1]}|\varphi(u)-\varphi_m(u)|^{2}.
\end{equation*}
Using Lemma~\ref{lem-exponential estimation},  for any $0\leq\eta<\frac{1}{2}$, we have
\begin{equation*}
\begin{aligned}
&\sup_{0\leq s<t\leq 1}
\mathbb E\exp\left\{\frac{\eta}{L^{2}\lambda^{2}(\epsilon)\sup_{u\in[0,1]}|\varphi(u)-\varphi_m(u)|^{2}}
\cdot\frac{\left|\mathcal{R}_{4,\epsilon}(t)-\mathcal{R}_{4,\epsilon}(s)\right|^2}{t-s}\right\}\\
&\leq\sup_{0\leq s<t\leq 1}
\mathbb E\exp\left\{\frac{\eta\left|\mathcal{R}_{4,\epsilon}(t)-\mathcal{R}_{4,\epsilon}(s)\right|^2}
{\lambda^{2}(\epsilon)\int^{t}_{s}|\sigma(\varphi(u))-\sigma(\varphi_m(u))|^{2}du}\right\}\\
&\leq\sup_{f\in L^{2}([0,1]),\|f\|_{L^{2}}\neq0}
\mathbb E\exp\bigg\{\frac{\eta\left|\int^{1}_{0}f(r)\theta_{\epsilon}(r)dr\right|^2}{\int_0^1f^2(r)dr}\bigg\}
<\infty.
\end{aligned}
\end{equation*}
Applying Theorem A.19 in Friz and Victoir  (\cite{FV-2010}) with $p=2, \zeta(h)=h^{1/3}$, there exists a uniform positive constant $c$, independent of ~$m$ and $\epsilon$ such that
\begin{equation*}
\begin{aligned}
\mathbb E\exp\left\{\frac{c\eta}{L^{2}\lambda^{2}(\epsilon)\sup_{u\in[0,1]}|\varphi(u)-\varphi_m(u)|^{2}}
\cdot\sup_{0\leq s\leq t\leq 1}\frac{\left|\mathcal{R}_{4,\epsilon}(t)-\mathcal{R}_{4,\epsilon}(s)\right|^2}{(t-s)^{1/3}}\right\}
<\infty.
\end{aligned}
\end{equation*}

Note that
$$
\sup_{t\in[0,1]}|\mathcal{R}_{4,\epsilon}(t)|
\leq
\sup_{0\leq s\leq t\leq 1}\frac{\left|\mathcal{R}_{4,\epsilon}(t)-\mathcal{R}_{4,\epsilon}(s)\right|}{(t-s)^{1/6}},
$$
and then
\begin{align*}
&\mathbb P\Big(\sup_{t\in[0,1]}|\mathcal{R}_{4,\epsilon}(t)|\geq\sup_{u\in[0,1]}|\varphi(u)-\varphi_m(u)|^{1/2}\Big)\\
&\leq\mathbb P\Big(\sup_{0\leq s\leq t\leq 1}\frac{\left|\mathcal{R}_{4,\epsilon}(t)-\mathcal{R}_{4,\epsilon}(s)\right|}{(t-s)^{1/6}}
\geq\sup_{u\in[0,1]}|\varphi(u)-\varphi_m(u)|^{1/2}\Big)\\
&\leq
\mathbb E\exp\left\{\frac{c\eta}{L^{2}\lambda^{2}(\epsilon)\sup_{u\in[0,1]}|\varphi(u)-\varphi_m(u)|^{2}}
\cdot\sup_{0\leq s\leq t\leq 1}\frac{\left|\mathcal{R}_{4,\epsilon}(t)-\mathcal{R}_{4,\epsilon}(s)\right|^2}{(t-s)^{1/3}}\right\}\\
&\quad\cdot\exp\left\{\frac{-c\eta}{L^{2}\lambda^{2}(\epsilon)\sup_{u\in[0,1]}|\varphi(u)-\varphi_m(u)|}\right\}.
\end{align*}
Therefore,
\begin{align*}
&\lim_{m\to\infty}\lim_{\epsilon\to0}\lambda^2(\epsilon)
\log\mathbb P\Big(\sup_{t\in[0,1]}|\mathcal{R}_{4,\epsilon}(t)|\geq\sup_{u\in[0,1]}|\varphi(u)-\varphi_m(u)|^{1/2}\Big)\\
&\leq-\lim_{m\to\infty}\frac{c\eta}{L^{2}\sup_{t\in[0,1]}|\varphi(t)-\varphi_m(t)|}=-\infty.
\end{align*}
That is to say,
\begin{equation}\label{eq4-negligible-stoppingtime}
\lim_{m\to\infty}\lim_{\epsilon\to0}\lambda^2(\epsilon)\log\mathbb P\Big(\chi_4<1\Big)
=-\infty.
\end{equation}

Combining (\ref{eq1-negligible-stoppingtime}), (\ref{eq2-negligible-stoppingtime-0}), (\ref{eq3-negligible-stoppingtime}) and (\ref{eq4-negligible-stoppingtime}),
we immediately obtain \eqref{negligible-stoppingtime} and conclude the proof of the lower bound. \\

{\bf (2) Upper bound.}  Thanks to the exponential tightness, we only need to prove local upper bound, more precisely,  for any $\varphi \in C([0,1])$,
\begin{equation}\label{upper bound}
\begin{split}
\lim_{\delta\to0}\limsup_{\epsilon\to0}\lambda^2(\epsilon)\log\mathbb P\big(X^{\epsilon}\in B(\varphi,\delta)\big) \leq -\mathcal{I}(\varphi).
\end{split}
\end{equation}

Similar to \eqref{difference-X-tildeX}, we can write
\begin{equation*}
\begin{aligned}
\hat{X}^{\epsilon}(t)-X^{\epsilon}(t)=\mathcal{\tilde{R}}_{1,\epsilon}(t)+\mathcal{\tilde{R}}_{2,\epsilon}(t)+\mathcal{R}_{4,\epsilon}(t),
\end{aligned}
\end{equation*}
where $\mathcal{R}_{4,\epsilon}$ is given in \eqref{difference-X-tildeX} and
\begin{align*}
&\mathcal{\tilde{R}}_{1,\epsilon}(t)=\int^{t}_{0}\big(b(\varphi(s))-b(X^{\epsilon}(s)))\big)ds,\\
&\mathcal{\tilde{R}}_{2,\epsilon}(t)=\lambda(\epsilon)\int^{t}_{0}\big(\sigma(\varphi_{m}(s))-\sigma(X^{\epsilon}(s))\big)\theta_{\epsilon}(s)ds.
\end{align*}
For any $\delta>0$,  define the following stopping times:
$$
\tilde{\chi}_i=\inf\Big\{t\geq0; |\mathcal{\tilde{R}}_{i,\epsilon}(t)|\geq \delta^{1/2}\Big\}, \quad
\tilde{\chi}=\min\{\tilde{\chi}_{1}, \tilde{\chi}_2, \chi_4\}, \quad i=1,2.
$$
Consequently, we can get
\begin{align*}
&\Big\{\sup_{t\in[0,1]}|X^{\epsilon}(t)-\varphi(t)|<\delta\Big\}\\
&=\Big\{\sup_{t\in[0,1]}|X^{\epsilon}(t)-\varphi(t)|<\delta, \tilde{\chi}<1\Big\}
\cup\Big\{\sup_{t\in[0,1]}|X^{\epsilon}(t)-\varphi(t)|<\delta, \tilde{\chi} \ge 1\Big\}\\
&\subset\Big\{\sup_{t\in[0,1]}|X^{\epsilon}(t)-\varphi(t)|<\delta, \tilde{\chi}<1\Big\}\\
&\quad\cup\Big\{\sup_{t\in[0,1]}|X^{\epsilon}(t)-\varphi(t)|<\delta, \sup_{t\in[0,1]}|\hat{X}^{\epsilon}(t)-X^{\epsilon}(t)|\leq2\delta^{1/2}+\sup_{t\in[0,1]}|\varphi(t)-\varphi_m(t)|^{1/2}\Big\}\\
&\subset\Big\{\sup_{t\in[0,1]}|X^{\epsilon}(t)-\varphi(t)|<\delta, \tilde{\chi}<1\Big\}\\
&\quad\cup\Big\{\sup_{t\in[0,1]}|\hat{X}^{\epsilon}(t)-\varphi(t)|<2\delta^{1/2}+\delta+\sup_{t\in[0,1]}|\varphi(t)-\varphi_m(t)|^{1/2}\Big\}.
\end{align*}
We make a claim whose proof will be given later:
\begin{equation}\label{upper-negligible-stoppingtime}
\lim_{\delta\to0}\lim_{m\to\infty}\limsup_{\epsilon\to0}\lambda^2(\epsilon)
\log\mathbb P\Big(\sup_{t\in[0,1]}|X^{\epsilon}(t)-\varphi(t)|<\delta, \tilde{\chi}<1\Big)
=-\infty.
\end{equation}
Then, we have
\begin{equation*}
\begin{split}
&\lim_{\delta\to0}\limsup_{\epsilon\to0}\lambda^2(\epsilon)\log\mathbb P\big(X^{\epsilon}\in B(\varphi,\delta)\big)\\
& \leq \lim_{\delta\to0}\lim_{m\to\infty}\limsup_{\epsilon\to0}\lambda^2(\epsilon)\log\mathbb P\Big(\hat{X}^{\epsilon} \in B\big(\varphi,2\delta^{1/2}+\delta+\sup_{t\in[0,1]}|\varphi(t)-\varphi_m(t)|^{1/2}\big)\Big)  \\
& \leq -\lim_{\delta\to0}\lim_{m\to\infty}\inf_{f \in B\big(\varphi, 2\delta^{1/2}+\delta+\sup_{t\in[0,1]}|\varphi(t)-\varphi_m(t)|^{1/2}\big)}
 \mathcal{I}^{\varphi}(f).
\end{split}
\end{equation*}
Since $\mathcal{I}^{\varphi}$ is lower semicontinuous, it holds
$$
\lim_{\delta\to0}\lim_{m\to\infty}\inf_{f \in B(\varphi, 2\delta^{1/2}+\delta+\|\varphi-\varphi_m\|^{1/2}_{\infty})} \mathcal{I}^{\varphi}(f)=\mathcal{I}^{\varphi}(\varphi)=\mathcal{I}(\varphi),
$$
which completes the proof of \eqref{upper bound}.

Finally, we turn to the proof of \eqref{upper-negligible-stoppingtime}. Obviously,
\begin{align*}
&\mathbb P\Big(\sup_{t\in[0,1]}|X^{\epsilon}(t)-\varphi(t)|<\delta, \tilde{\chi}<1\Big)\\
&\leq\mathbb P\Big(\sup_{t\in[0,1]}|X^{\epsilon}(t)-\varphi(t)|<\delta, \tilde{\chi}_{1}<1\Big)
+\mathbb P\Big(\sup_{t\in[0,1]}|X^{\epsilon}(t)-\varphi(t)|<\delta, \tilde{\chi}_{2}<1\Big)
+\mathbb P\Big(\chi_{4}<1\Big).
\end{align*}
Since the estimation of $\mathbb P\Big(\chi_{4}<1\Big)$ has been given by (\ref{eq4-negligible-stoppingtime}),
it remains to consider
$$
\mathbb P\Big(\sup_{t\in[0,1]}|X^{\epsilon}(t)-\varphi(t)|<\delta, \tilde{\chi}_{i}<1\Big),\quad i=1,2.
$$

On the event $\Big\{\sup_{t\in[0,1]}|X^{\epsilon}(t)-\varphi(t)|<\delta\Big\}$, we observe
$$
\sup_{t\in[0,1]}|\mathcal{\tilde{R}}_{1,\epsilon}(t)|\leq L\delta<\delta^{1/2}
$$
as $\delta$ is sufficiently small.
So the set $\{\sup_{t\in[0,1]}|X^{\epsilon}(t)-\varphi(t)|<\delta, \tilde{\chi}_1<1\}$ is empty and we have
\begin{equation}\label{eq1-negligible-stoppingtime-tilde}
\lim_{\delta\to0}\lim_{\epsilon\to0}\lambda^2(\epsilon)
\log\mathbb P\Big(\sup_{t\in[0,1]}|X^{\epsilon}(t)-\varphi(t)|<\delta, \tilde{\chi}_1<1\Big)
=-\infty.
\end{equation}

Moreover, from (\ref{Theta-theta-relation}), it follows that
\begin{equation*}
\begin{aligned}
\mathcal{\tilde{R}}_{2,\epsilon}(t)
&=-\lambda(\epsilon)\epsilon^{2}\int^{t}_{0}\big(\sigma(\varphi_{m}(s))-\sigma({X}^{\epsilon}(s-))\big)
\theta_{\epsilon}(s-)\tilde{N}^{\epsilon^{-2}}(ds)\\
&\quad-\frac{\lambda(\epsilon)\epsilon^{2}}{2}\int^{t}_{0}\big(\sigma(\varphi_{m}(s))-\sigma({X}^{\epsilon}(s))\big)d\theta_{\epsilon}(s).
\end{aligned}
\end{equation*}
By the same argument for showing (\ref{e:Tight-Bound}) and (\ref{estimate-mathcalR2}),
 on the event of $\big\{\sup_{t\in[0,1]}|X^{\epsilon}(t)-\varphi(t)|<\delta\big\}$, we obtain
\begin{equation}\label{upper-estimate-mathcalR2}
\begin{aligned}
|\mathcal{\tilde{R}}_{2,\epsilon}(t)|
&\leq\vartheta\lambda^{2}(\epsilon)\int^{t}_{0}1+|{X}^{\epsilon}(s)|+|\varphi'_m(s)|ds
\\
&+\lambda(\epsilon)\epsilon^{2}\Big|\int^{t}_{0}\big(\sigma(\varphi_{m}(s))-\sigma(X^{\epsilon}(s-))\big)
\theta_{\epsilon}(s-)\tilde{N}^{\epsilon^{-2}}(ds)\Big|\\
&\leq\vartheta\lambda^{2}(\epsilon)\int^{t}_{0}2+|\varphi(s)|+|\varphi'_m(s)|ds
\\
&+\lambda(\epsilon)\epsilon^{2}\Big|\int^{t}_{0}\big(\sigma(\varphi_{m}(s))-\sigma(X^{\epsilon}(s-))\big)
\theta_{\epsilon}(s-)\tilde{N}^{\epsilon^{-2}}(ds)\Big|
\end{aligned}
\end{equation}
where $\vartheta$ is a sufficiently large positive constant depending only on~$L$,~$b(0)$, $\sigma(0)$, $b(x_{0}),\sigma(x_{0})$ and $\sigma(\varphi_{m}(0))$.
Define the following stopping time
$$
\zeta=\inf\Big\{t\geq0; |X^{\epsilon}(t)-\varphi(t)|\geq\delta\Big\}.
$$
For any $\eta>0$, it holds by \eqref{upper-estimate-mathcalR2}
\begin{align*}
&\lim_{\epsilon\to0}\lambda^2(\epsilon)
\log\mathbb P\Big(\sup_{t\in[0,1]}|X^{\epsilon}(t)-\varphi(t)|<\delta, \tilde{\chi}_2<1\Big)\\
&\leq\lim_{\epsilon\to0}\lambda^2(\epsilon)
\log\mathbb P\Big(\sup_{t\in[0,\zeta\wedge1]}|\mathcal{\tilde{R}}_{2,\epsilon}(t)|\geq\delta^{1/2}\Big)\\
&\leq\lim_{\epsilon\to0}\lambda^2(\epsilon)
\log\mathbb P\Big(\sup_{t\in[0,\zeta\wedge1]}
\lambda(\epsilon)\epsilon^{2}\Big|\int^{t}_{0}\big(\sigma(\varphi_{m}(s))-\sigma(X^{\epsilon}(s-))\big)
\theta_{\epsilon}(s-)\tilde{N}^{\epsilon^{-2}}(ds)\Big|\geq\frac{1}{2}\delta^{1/2}\Big)\\
&\leq -\frac{1}{2}\eta\delta^{1/2}+\frac{\eta^2L^{2}}{2}\big(\delta+\sup_{t\in[0,1]}|\varphi_m(t)-\varphi(t)|\big)^2
\lim_{\epsilon\to0}e^{\epsilon\lambda^{-1}(\epsilon)\eta L\big(\delta+\sup_{t\in[0,1]}|\varphi_m(t)-\varphi(t)|\big)},
\end{align*}
where the last inequality is obtained by  Proposition \ref{exponential inequality} with
$$
0\leq\iota_2-\iota_1\leq1,\quad \rho=L(\delta+\sup_{t\in[0,1]}|\varphi_m(t)-\varphi(t)|),\quad a=\frac{1}{2}\delta^{1/2}.
$$
Consequently, it holds
\begin{align*}
&\lim_{m\to\infty}\lim_{\epsilon\to0}\lambda^2(\epsilon)
\log\mathbb P\Big(\sup_{t\in[0,1]}|X^{\epsilon}(t)-\varphi(t)|<\delta, \tilde{\chi}_2<1\Big)\\
&\leq \left\{\begin{array}{ll}
-\frac{1}{2}\eta\delta^{1/2}+\frac{\eta^2 L^{2}}{2}\delta^2, & \lim_{\epsilon\to0}\frac{\lambda(\epsilon)}{\epsilon}=+\infty;\\
-\frac{1}{2}\eta\delta^{1/2}+\frac{\eta^2 L^{2}}{2}\delta^2
e^{\eta\delta L/{\kappa}}, & \lim_{\epsilon\to0}\frac{\lambda(\epsilon)}{\epsilon}=\kappa.\\\end{array}\right.
\end{align*}
Choosing $\eta=\delta^{-1}$, we have
\begin{equation*}\label{eq2-negligible-stoppingtime}
\lim_{\delta\to0}\lim_{m\to\infty}\lim_{\epsilon\to0}\lambda^2(\epsilon)
\log\mathbb P\Big(\sup_{t\in[0,1]}|X^{\epsilon}(t)-\varphi(t)|<\delta, \tilde{\chi}_2<1\Big)
=-\infty,
\end{equation*}
which together with (\ref{eq1-negligible-stoppingtime-tilde}) implies \eqref{upper-negligible-stoppingtime}.




\section{Appendix} \label{s:a}
In this section, we first show the proof of Proposition \ref{functional LDP} explicitly. Then, some estimations of tail probability for small noise diffusions and
deviation inequalities for Poisson random integral will be given.

\subsection{Proof of Proposition \ref{functional LDP}}

The proof is standard and we give the proof for completeness. Now, we split the proof of this theorem into the following three steps.

\emph{{\bf Step 1}. LDP for $\epsilon\Theta_{\epsilon}(1)=\epsilon^2\int_0^{\epsilon^{-2}} \xi(s)ds$ and the related Cram\'er function.} It seems not easy to use G\"artner-Ellis Theorem to prove this LDP, alternatively we first consider the LDP of empirical process $\epsilon^2\int_0^{\epsilon^{-2}}\delta_{\xi(s)}ds$ and then use contraction principle.

It is easy to see that $\xi(t)=(-1)^{N(t)}$ is a ~$\{-1,1\}$-valued reversible Markov process with stationary distribution $\mu$ satisfying
$$
\mu(1)=\mu(-1)=\frac{1}{2}.
$$
and Q-matrix
$$
Q=\left(\begin{array}{ll}-1\\
~~ 1\\\end{array}
\begin{array}{ll}~~ 1\\
-1\\\end{array}\right).
$$
By \cite[Theorem 4.2.58]{DS-1989}, the empirical measure
$$
\Big\{\epsilon^2\int_0^{\epsilon^{-2}}\delta_{\xi(s)}ds,\quad \epsilon>0\Big\}
$$
satisfies the large deviations with speed $\epsilon^{-2}$ and  rate function $L$ defined by
\begin{equation*}
\begin{aligned}
L(\nu)=\left\{\begin{array}{ll}
<f^{1/2}, -Qf^{1/2}>_{\mu}, &\textrm{if } f=\frac{d\nu}{d\mu};\\
+\infty, & \textrm{otherwise }.\\\end{array}\right.
\end{aligned}
\end{equation*}
Straightforward calculation yields that
\begin{equation}  \label{e:LForm}
L(\nu)=(\sqrt{\nu(1)}-\sqrt{\nu(-1)})^{2}
\end{equation}
with
\begin{equation} \label{e:NuForm}
 \quad \nu(1)+\nu(-1)=1, \ \ \ \ 0\leq\nu(1)\leq1.
\end{equation}
Since $\xi(t)$ is $\{-1,+1\}$-valued, by contraction principle,
$$
\Big\{\epsilon^{2}\int^{\epsilon^{-2}}_{0}\xi(s)ds, \epsilon>0\Big\}
$$
satisfies the large deviation with speed ~$\epsilon^{-2}$ and rate function: for $x \in \mathbb R$,
$$
\Lambda^*(x)=\inf\Big\{L(\nu), \int_{\mathbb R} y\nu(dy)=x\Big\}=\inf\Big\{L(\nu),\nu(1)-\nu(-1)=x\Big\}.
$$
This, together with \eqref{e:NuForm}, implies
$$\nu(1)=\frac{1+x}{2}, \ \ \ \ \ \nu(-1)=\frac{1-x}2, \ \ \ \ \ |x| \le 1,$$
and that $\nu$ is not a probability measure as $|x|>1$. Therefore,
$$
\Lambda^*(x)=\infty, \ \ \ \ |x|>1,
$$
and
$$
\Lambda^*(x)=\left(\sqrt{\frac{1+x}{2}}-\sqrt{\frac{1-x}{2}}\right)^2=1-\sqrt{1-x^2}, \ \ \ \ \ |x| \le 1.
$$

For any ~$\alpha\in\mathbb{R}$, define
$$
\Lambda_{\epsilon^2}(\alpha)=\epsilon^2\log\mathbb E\Big(e^{\alpha\int^{\epsilon^{-2}}_{0}\xi(s)ds}\Big),\quad
\Lambda(\alpha)=\sup_{x\in\mathbb{R}}\Big\{\alpha x-\Lambda^*(x)\Big\}.
$$
Since ~$\Lambda^*$ is a good rate function, by using Varadhan's integral lemma \cite[Theorem 4.3.1]{DZ-1998}, we know
\begin{equation}\label{Lambda-n-Lambda-0}
\Lambda(\alpha)=\lim_{\epsilon\to0}\Lambda_{\epsilon^2}(\alpha).
\end{equation}
Since $\Lambda^*$ is even,
$$\Lambda(\alpha)=\Lambda(-\alpha), \ \ \ \ \alpha \in \mathbb R.$$

\emph{{\bf Step~2.} \ LDP for $(\epsilon\Theta_{\epsilon}(t_1),\cdots,\epsilon\Theta_{\epsilon}(t_k))$ with~$0\leq t_1<t_2<\cdots<t_k\leq1$}. 
Denote $\alpha=(\alpha_1,\cdots, \alpha_k)\in\mathbb{R}^k$, let us consider
$$
\Lambda_{\epsilon^2,t_1,\cdots t_k}(\alpha):=\epsilon^2\log \mathbb E\exp\left(\epsilon^{-1}\sum_{i=1}^k\alpha_i\Theta_{\epsilon}(t_i)\right).
$$
By \eqref{Main term pro} we have
\begin{equation*}
\begin{aligned}
\Lambda_{\epsilon^2,t_1,\cdots t_k}(\alpha)&=\epsilon^2\log \mathbb E \exp\left(\sum_{i=1}^k\alpha_i\int_0^{\epsilon^{-2}t_i}\xi(s)ds\right).
\end{aligned}
\end{equation*}
Observing
$$\sum_{i=1}^k\alpha_i\int_0^{\epsilon^{-2}t_i}\xi(s)ds=\sum^{k}_{i=1}\sum^{i}_{\ell=1}\alpha_i\int^{\epsilon^{-2}t_{\ell}}_{\epsilon^{-2}t_{\ell-1}}\xi(s)ds=\sum^{k}_{\ell=1} \bar\alpha_{\ell,k}\int^{\epsilon^{-2}t_{\ell}}_{\epsilon^{-2}t_{\ell-1}}\xi(s)ds$$
with $\bar \alpha_{\ell,k}=\sum^{k}_{i=\ell}\alpha_{i}$
and writing $\Delta_\ell=\bar \alpha_{\ell,k} \int^{\epsilon^{-2}t_{\ell}}_{\epsilon^{-2}t_{\ell-1}}\xi(s)ds$, we have
\begin{equation} \label{e:SumKDel}
\begin{aligned}
\mathbb E \exp\left(\sum_{i=1}^k\alpha_i\int_0^{\epsilon^{-2}t_i}\xi(s)ds\right)
&=\mathbb E \exp\left(\sum^{k}_{\ell=1}\Delta_\ell\right)\\
&=\mathbb E \left[\exp\left(\sum^{k-1}_{\ell=1}\Delta_\ell\right) \mathbb E\left(e^{\Delta_k}\Big|\mathcal{F}_{\epsilon^{-2} t_{k-1}}\right)\right],
\end{aligned}
\end{equation}
where~$\mathcal{F}_t=\sigma\big(N(s), 0\leq s\leq t\big)$. Moreover, by Markov property we have
\begin{equation*}
\begin{aligned}
\mathbb E\left(e^{\Delta_k}\Big|\mathcal{F}_{t_{k-1}}\right)
&=\mathbb E\left[\exp\left(\bar \alpha_{k,k}(-1)^{N(\epsilon^{-1}t_{k-1})}\cdot\int^{\epsilon^{-2}t_{k}}_{\epsilon^{-2}t_{k-1}}
(-1)^{N(s)-N(\epsilon^{-2}t_{k-1})}ds\right)\Big|\mathcal{F}_{\epsilon^{-2} t_{k-1}}\right]\\
&=\mathbb E\left[\exp\left(\bar \alpha_{k,k} (-1)^{N(\epsilon^{-2}t_{k-1})}\cdot\int^{\epsilon^{-2}t_{k}}_{\epsilon^{-2}t_{k-1}}
(-1)^{N(s)-N(\epsilon^{-2}t_{k-1})}ds\right)\Big|N(\epsilon^{-2}t_{k-1})\right]\\
&=1_{\{(-1)^{N(\epsilon^{-2} t_{k-1})}=1\}}\mathbb E\exp\left(\bar \alpha_{k,k}\int^{\epsilon^{-2}(t_{k}-t_{k-1})}_{0}\xi(s) ds\right)\\
&\quad+1_{\{(-1)^{N(\epsilon^{-2} t_{k-1})}=-1\}}\mathbb E\exp\left(-\bar \alpha_{k,k} \int^{\epsilon^{-2}(t_{k}-t_{k-1})}_{0}\xi(s)ds\right)\\
&=1_{\{(-1)^{N(t_{k-1})}=1\}}\exp\left(\epsilon^{-2}(t_{k}-t_{k-1})\Lambda_{\epsilon^{2}/(t_{k}-t_{k-1})}(\bar \alpha_{k,k})\right)\\
&\quad+1_{\{(-1)^{N(t_{k-1})}=-1\}}\exp\left(\epsilon^{-2}(t_{k}-t_{k-1})\Lambda_{\epsilon^{2}/(t_{k}-t_{k-1})}(-\bar \alpha_{k,k})\right). \\
&=1_{\{(-1)^{N(t_{k-1})}=1\}}\exp\left(\epsilon^{-2}(t_{k}-t_{k-1}) \mathcal{E}_{\epsilon,\bar \alpha_{k,k}}\right)\exp\left(\epsilon^{-2}(t_{k}-t_{k-1}) \Lambda(\bar \alpha_{k,k})\right)\\
&\quad+1_{\{(-1)^{N(t_{k-1})}=-1\}}\exp\left(\epsilon^{-2}(t_{k}-t_{k-1}) \mathcal{E}_{\epsilon,-\bar \alpha_{k,k}}\right)\exp\left(\epsilon^{-2}(t_{k}-t_{k-1}) \Lambda(\bar \alpha_{k,k})\right),
\end{aligned}
\end{equation*}
where $\mathcal{E}_{\epsilon,\alpha}=\Lambda_{\epsilon^{2}/(t_{i}-t_{i-1})}({\alpha})-\Lambda({\alpha})$ and we have used the relation $\Lambda(\alpha)=\Lambda(-\alpha)$.

By (\ref{Lambda-n-Lambda-0}), for any ~$\delta>0$, when $\epsilon$ is small enough, we have
$$
\Big|\mathcal{E}_{\epsilon,\pm\bar \alpha_{k,k}}\Big| \le \delta,
$$
thus
\begin{equation}
\begin{split}
\exp\bigg[\epsilon^{-2}(t_{k}-t_{k-1})(\Lambda(\bar \alpha_{k,k})-\delta)\bigg] \le \mathbb E\left(e^{\Delta_k}\Big|\mathcal{F}_{t_{k-1}}\right)  \le \exp \bigg[\epsilon^{-2}(t_{k}-t_{k-1})(\Lambda(\bar \alpha_{k,k})+\delta)\bigg].
\end{split}
\end{equation}
Recalling \eqref{e:SumKDel} and applying the same argument to
$\mathbb E \left[\exp\left(\sum^{j}_{\ell=1}\Delta_\ell\right)\right]$ ($j=k-1,k-2,...,1$) with an induction, we finally obtain
\begin{equation*} 
\begin{split}
\exp\bigg[\epsilon^{-2}\sum_{\ell=1}^k (t_{\ell}-t_{\ell-1}) (\Lambda(\bar \alpha_{\ell,k}-\delta) \bigg]\le \mathbb E \exp\left(\sum^{k}_{\ell=1}\Delta_\ell\right) \le \exp\bigg[\epsilon^{-2}\sum_{\ell=1}^k (t_{\ell}-t_{\ell-1}) (\Lambda(\bar \alpha_{\ell,k}+\delta) \bigg]
\end{split}
\end{equation*}
for sufficiently small $\epsilon$. Letting $\epsilon \rightarrow 0$ and
$\delta \rightarrow 0$ in order, we have
\begin{equation*}
\begin{aligned}
\Lambda_{t_1,\cdots t_k}(\alpha):=\lim_{\epsilon \rightarrow 0} \Lambda_{\epsilon^2,t_1,\cdots t_k}(\alpha)=\sum_{\ell=1}^k (t_{\ell}-t_{\ell-1}) \Lambda(\bar \alpha_{\ell,k}).
\end{aligned}
\end{equation*}

By Fenchel-Legendre transform, we get
\begin{equation*}
\begin{aligned}
\Lambda^{*}_{t_{1},\cdots,t_{k}}(x)
:&=\sup_{\alpha\in\mathbb{R}^k}\Big\{\sum^{k}_{i=1} \alpha_i x_i-\sum^{k}_{i=1}(t_{i}-t_{i-1})\Lambda(\bar{\alpha}_{i,k})\Big\}\\
&=\sup_{\alpha\in\mathbb{R}^k}\Big\{\sum^{k}_{i=1}(x_i-x_{i-1})\bar{\alpha}_{i,k}-\sum^{k}_{i=1}(t_{i}-t_{i-1})\Lambda(\bar{\alpha}_{i,k})\Big\} \\
&=\sup_{\alpha\in\mathbb{R}^k}\Big\{\sum^{k}_{i=1}(x_i-x_{i-1})\alpha_{i}-\sum^{k}_{i=1}(t_{i}-t_{i-1})\Lambda({\alpha}_{i})\Big\}\\
&=\sum^{k}_{i=1}(t_{i}-t_{i-1})\Lambda^{*}\left(\frac{x_{i}-x_{i-1}}{t_{i}-t_{i-1}}\right),
\end{aligned}
\end{equation*}
where~$x=(x_1,\cdots, x_k)^{\tau}\in\mathbb{R}^k$ and $x_{0}=0$.
Consequently, from G\"{a}rtner-Ellis Theorem, it follows that
$$
\Big(\epsilon\Theta_{\epsilon}(t_1),\cdots,\epsilon\Theta_{\epsilon}(t_k)\Big)^{\tau}
$$
satisfies the large deviations with speed ~$\epsilon^2$ and rate function~$\Lambda^{*}_{t_{1},\cdots,t_{k}}(x)$

{\bf{Step~3}}  To establish functional large deviations for $\{\epsilon \Theta_{\epsilon}(t), t\in[0,1]\}$ equipped with the uniform topology.
Indeed, by Theorem 4.6.9 and following the argument of Lemma ~5.1.6 and Lemma~5.1.8 in Dembo and Zetouni~(\cite{DZ-1998}),
the family of processes $\{\epsilon \Theta_{\epsilon}(t), t\in[0,1]\}$
satisfies the large deviations with speed ~$\epsilon^2$ and  rate function $J$ defined by~(\ref{rate function-main result}),
in ~$C([0,1])$ equipped with the pointwise convergence topology. Moreover, by using the fact that
$$
\big|\epsilon \Theta_{\epsilon}(t)-\epsilon \Theta_{\epsilon}(s)\big|\leq|t-s|,
$$
 we can get  for any $a>0$,
\begin{align*}
&\limsup_{\delta\to0}\limsup_{\epsilon\to0}\epsilon^{2}\log\mathbb P\Big(\sup_{|t-s|<\delta}
\big|\epsilon \Theta_{\epsilon}(t)-\epsilon \Theta_{\epsilon}(s)\big|>a\Big)
=-\infty,
\end{align*}
which implies the exponential tightness of  $\{\epsilon \Theta_{\epsilon}(t), t\in[0,1]\}$
in ~$C([0,1])$.
Therefore, the family of processes $\{\epsilon \Theta_{\epsilon}(t), t\in[0,1]\}$
satisfies the large deviations with speed ~$\epsilon^2$ and  rate function $J$ ,
in ~$C([0,1])$.

\subsection{Estimations of tail probability for small noise diffusions}
Recall that $N^{\alpha}$ is a Poisson random measure with the intensity measure $\alpha dt$ on ~$(\mathbb R_{+}, \mathcal B(\mathbb R_{+}))$
 and $\tilde{N}^{\alpha}(ds)={N}^{\alpha}(ds)-\alpha ds$ is the compensated Poisson random measure.
Now, as the main result of this subsection, the following proposition gives estimation of tail probability for small noise diffusions,
which plays a crucial role in the proof of Theorem \ref{thm-multiplicative noise}.

\begin{pro}\label{comparison est lem}
Assume that $U^{\epsilon}=\{U^{\epsilon}(t), t\in[0,1]\}$ is a continuous stochastic process valued in~$\mathbb{R}$.
Define  ~$\hat{Y}^{\epsilon}=\{\hat{Y}^{\epsilon}(t), t\in[0,1]\}$ and ~$\check{Y}^{\epsilon}=\{\check{Y}^{\epsilon}(t), t\in[0,1]\}$ as follows
\begin{equation*}
\begin{aligned}
&\hat{Y}^{\epsilon}(t)=\hat{y}^{\epsilon}_{0}+\int^{t}_{0}\hat{h}^{\epsilon}(s)ds
+\lambda(\epsilon)\epsilon\int^{t}_{0}\hat{\gamma}^{\epsilon}(s)\tilde{N}^{\epsilon^{-2}}(ds),
\end{aligned}
\end{equation*}
and
\begin{equation*}
\begin{aligned}
&\check{Y}^{\epsilon}(t)=\check{y}^{\epsilon}_{0}+\int^{t}_{0}\check{h}^{\epsilon}(s)ds
+\lambda(\epsilon)\epsilon\int^{t}_{0}\check{\gamma}^{\epsilon}(s)\tilde{N}^{\epsilon^{-2}}(ds).
\end{aligned}
\end{equation*}
For the stopping time $\tau\in[0,1]$ and some positive constants $\rho$, $B$ and $M$, suppose the following assumptions hold.

\noindent(1). For any $t\in[0,\tau]$,
\begin{align*}
&|\hat{h}^{\epsilon}(t)|\vee|\check{h}^{\epsilon}(t)|\leq B\left(\rho^{2}+|U^{\epsilon}(t-)|^{2}\right)^{\frac{1}{2}},\\
&|\hat{\gamma}^{\epsilon}(t)|\vee|\check{\gamma}^{\epsilon}(t)|\leq M\left(\rho^{2}+|U^{\epsilon}(t-)|^{2}\right)^{\frac{1}{2}}.
\end{align*}

\noindent(2). For any $t\in[0,\tau)$, $\check{Y}^{\epsilon}(t)\leq U^{\epsilon}(t)\leq\hat{Y}^{\epsilon}(t)$.

Take $\epsilon$ sufficient small such that $1-\epsilon\lambda(\epsilon)M>\frac{1}{\sqrt{2}}$. Then, for any $\delta>0$,
\begin{equation*}
\begin{aligned}
&\lambda^2(\epsilon)\log\mathbb P\Big(\sup_{t\in[0,\tau]}|U^{\epsilon}(t)|>\delta\Big)\\
&\leq 2\sqrt{2}B+4M^2\epsilon^{2}(2+3\lambda^{2}(\epsilon))e^{4M\epsilon/{\lambda(\epsilon)}}
+\log\Big(\frac{\rho^{2}+|\hat{y}^{\epsilon}_{0}|^{2}+|\check{y}^{\epsilon}_{0}|^{2}}{\rho^{2}+\delta^{2}}\Big).
\end{aligned}
\end{equation*}
\end{pro}

To prove Proposition \ref{comparison est lem}, we need the following lemma and its proof will be postponed to the end of this subsection.

\begin{lem}\label{Est lemma}
Define the $\mathbb{R}^{2}$-valued processes ~$Y^{\epsilon}=\{Y^{\epsilon}(t), t\in[0,1]\}$ as follows
\begin{equation}\label{def-Y-sde}
\begin{aligned}
Y^{\epsilon}(t)&=y^{\epsilon}_{0}+\int^{t}_{0}h^{\epsilon}(s)ds
+\lambda(\epsilon)\epsilon\int^{t}_{0}\gamma^{\epsilon}(s)\tilde{N}^{\epsilon^{-2}}(ds).
\end{aligned}
\end{equation}
For the stopping time ~$\tau\in[0,1]$, suppose that there exist some positive constants~$B$,~$\rho$ and~$M$ such that for any $t\in[0,\tau]$,
\begin{equation}\label{condition-b-gamma}
\begin{aligned}
&|h^{\epsilon}(t)|\leq B\left(\rho^{2}+|Y^{\epsilon}(t-)|^{2}\right)^{\frac{1}{2}},\\
&|\gamma^{\epsilon}(t)|\leq M\left(\rho^{2}+|Y^{\epsilon}(t-)|^{2}\right)^{\frac{1}{2}}.
\end{aligned}
\end{equation}

Take $\epsilon$ sufficient small such that $1-\epsilon\lambda(\epsilon)M>\frac{1}{\sqrt{2}}$. Then,  for any $\delta>0$
\begin{equation*}
\begin{aligned}
&\lambda^{2}(\epsilon)\log\mathbb P\Big(\sup_{t\in[0,\tau]}|Y^{\epsilon}(t)|>\delta\Big)\\
&\leq 2B+2M^2\epsilon^{2}(2+3\lambda^{2}(\epsilon))\cdot e^{2\sqrt{2}M\epsilon/{\lambda(\epsilon)}}+\log\Big(\frac{\rho^{2}+\|y^{\epsilon}_{0}\|^{2}}{\rho^{2}+\delta^{2}}\Big).
\end{aligned}
\end{equation*}
\end{lem}

\begin{proof}[\noindent\textbf{\emph{Proof of Proposition~\ref{comparison est lem}}}]
For any $t\in[0,\tau]$,
\begin{equation*}
\begin{aligned}
&|\hat{h}^{\epsilon}(t)|\vee|\check{h}^{\epsilon}(t)|\leq B\Big(\rho^{2}+|\hat{Y}^{\epsilon}(t-)|^{2}+|\check{Y}^{\epsilon}(t-)|^{2}\Big)^{\frac{1}{2}},\\
&|\hat{\gamma}^{\epsilon}(t)|\vee|\check{\gamma}^{\epsilon}(t)|
\leq M\Big(\rho^{2}+|\hat{Y}^{\epsilon}(t-)|^{2}+|\check{Y}^{\epsilon}(t-)|^{2}\Big)^{\frac{1}{2}}.
\end{aligned}
\end{equation*}
Define ~$Y^{\epsilon}(t)=\big(\hat{Y}^{\epsilon}(t), \check{Y}^{\epsilon}(t)\big)$.
Then, ~$Y^{\epsilon}$ satisfies the conditions of Lemma ~\ref{Est lemma} with $y^{\epsilon}_{0}=\big(\hat{y}^{\epsilon}_{0}, \check{y}^{\epsilon}_{0}\big)$.
Consequently, for the stopping time $\tau\in[0,1]$ and any $\delta>0$,
$$
\lambda^{2}(\epsilon)\log\mathbb P\Big(\sup_{t\in[0,\tau]}\|Y^{\epsilon}(t)\|>\delta\Big)
\leq 2\sqrt{2}B+4M^2\epsilon^{2}(2+3\lambda^{2}(\epsilon))e^{4M\epsilon/{\lambda(\epsilon)}}
+\log\Big(\frac{\rho^{2}+\|y^{\epsilon}_{0}\|^{2}}{\rho^{2}+\delta^{2}}\Big).
$$
Together with the fact that ~$\sup_{t\in[0,\tau]}|U^{\epsilon}(t)|^{2}\leq \sup_{t\in[0,\tau]}(|\hat{Y}^{\epsilon}(t)|^{2}+|\check{Y}^{\epsilon}(t)|^{2})=\sup_{t\in[0,\tau]}|Y^{\epsilon}(t)|^{2}$,
we can complete the proof of this proposition.
\end{proof}

To end this subsection, we will give the proof to Lemma \ref{Est lemma}.  Firstly, we give the following auxiliary result.

\begin{lem}\label{Elementary inequality}
Assume that $\rho\geq0$,  $0<\beta<1$, and $a,b\in\mathbb{R}^d$ satisfy
$$
|b|^{2}\leq \beta(\rho^{2}+|a|^{2}).
$$
Then we have
$$
\rho^{2}+|a|^{2}\leq \frac{\rho^{2}+|a+b|^{2}}{(1-\sqrt{\beta})^{2}}.
$$
\end{lem}
\begin{proof}
For any~$\epsilon>0$, it holds
\begin{align*}
|a|^{2}&\leq|a+b|^{2}+|b|^{2}+2|a+b||b|\\
&\leq|a+b|^{2}+|b|^{2}+\epsilon|a+b|^{2}+\frac{|b|^{2}}{\epsilon}\\
&\leq (1+\epsilon)|a+b|^{2}+\Big(1+\frac{1}{\epsilon}\Big)|b|^{2}\\
&\leq(1+\epsilon)|a+b|^{2}+\beta\Big(1+\frac{1}{\epsilon}\Big)(\rho^{2}+|a|^{2}),
\end{align*}
which implies immediately
$$
\rho^{2}+|a|^{2}\leq(1+\epsilon)(\rho^{2}+|a+b|^{2})+\beta\Big(1+\frac{1}{\epsilon}\Big)(\rho^{2}+|a|^{2}).
$$
Therefore, provided that $\epsilon(1-\beta)-\beta>0$, we obtain
$$
\rho^{2}+|a|^{2}\leq\frac{\epsilon(1+\epsilon)}{\epsilon(1-\beta)-\beta}\Big(\rho^{2}+|a+b|^{2}\Big).
$$
Taking $\epsilon=\frac{\beta+\sqrt{\beta}}{1-\beta}$, we have
$$
\rho^{2}+|a|^{2}
\leq\frac{(1+\sqrt{\beta})^{2}}{(1-\beta)^{2}}\left(\rho^{2}+|a+b|^{2}\right)
=\frac{\rho^{2}+|a+b|^{2}}{(1-\sqrt{\beta})^{2}}.
$$
\end{proof}

\begin{proof}[\noindent\textbf{\emph{Proof of Lemma~\ref{Est lemma}}}]
Define for~$u\in\mathbb{R}^d$
$$
\Phi(u)=\log(\rho^{2}+|u|^{2}),\quad
\psi(u)=\exp\Big\{\frac{\Phi(u)}{\lambda^{2}(\epsilon)}\Big\}=(\rho^{2}+|u|^{2})^{\frac{1}{\lambda^2(\epsilon)}}.
$$
By It\^{o}'s formula,  we can write
\begin{equation}\label{eq-0-proof-Est lemma}
\begin{aligned}
\psi(Y^{\epsilon}(t))
&=\psi(y^{\epsilon}_{0})+\int_0^t\mathcal{B}^{\epsilon}(s)\tilde{N}^{\epsilon^{-2}}(ds)\\
&\quad+\int_0^t\langle\nabla\psi(Y^{\epsilon}(s-)),h^{\epsilon}(s)\rangle ds
+\int_0^t\mathcal{A}^{\epsilon}(s)ds,
\end{aligned}
\end{equation}
where
\begin{equation}\label{eq-1-proof-Est lemma}
\begin{aligned}
\mathcal{A}^{\epsilon}(s)&=\exp\Big\{\frac{\Phi(Y^{\epsilon}(s-)+\lambda(\epsilon)\epsilon\gamma^{\epsilon}(s))}{\lambda^2(\epsilon)}\Big\}-\exp\Big\{\frac{\Phi(Y^{\epsilon}(s-)}{\lambda^2(\epsilon)}\Big\}\\
&\quad-\frac{\epsilon\langle\nabla\Phi(Y^{\epsilon}(s-),\gamma^{\epsilon}(s)\rangle}{\lambda(\epsilon)}
\exp\Big\{\frac{\Phi(Y^{\epsilon}(s-)}{\lambda^2(\epsilon)}\Big\}\\
\end{aligned}
\end{equation}
and
$$
\mathcal{B}^{\epsilon}(s)=\exp\Big\{\frac{\Phi(Y^{\epsilon}(s-)+\lambda(\epsilon)\epsilon\gamma^{\epsilon}(s))}{\lambda^2(\epsilon)}\Big\}
-\exp\Big\{\frac{\Phi(Y^{\epsilon}(s-)}{\lambda^2(\epsilon)}\Big\}.
$$

Firstly,  Taylor formula gives
\begin{equation}\label{eq-2-proof-Est lemma}
\begin{aligned}
\mathcal{A}^{\epsilon}(s)
&=\frac{\epsilon^{2}}{2\lambda^2(\epsilon)}\psi(\tilde{Y}^{\epsilon}(s))\langle\nabla\Phi(\tilde{Y}^{\epsilon}(s)\otimes\nabla\Phi(\tilde{Y}^{\epsilon}(s),
\gamma^{\epsilon}(s)\otimes\gamma^{\epsilon}(s)\rangle\\
&\quad-\frac{\epsilon^{2}}{2}\psi(\tilde{Y}^{\epsilon}(s))\langle\nabla\otimes\nabla\Phi\big(\tilde{Y}^{\epsilon}(s)\big),
\gamma^{\epsilon}(s)\otimes\gamma^{\epsilon}(s)\rangle,
\end{aligned}
\end{equation}
where~$\tilde{Y}^{\epsilon}(s)=Y^{\epsilon}(s-)+\theta\lambda(\epsilon)\epsilon\gamma(s)$,
$\theta\in[0,1]$. Now, we will give the estimates of above terms respectively.

Let~$a=Y^{\epsilon}(s-)$, ~$b=\theta\lambda(\epsilon)\epsilon\gamma^{\epsilon}(s)$. For $s\in[0,\tau]$,
 we have by (\ref{condition-b-gamma})
$$
|b|^{2}\leq\lambda^2(\epsilon)\epsilon^{2}\|\gamma^{\epsilon}(s)\|^{2}\leq M^2\epsilon^{2}\lambda^2(\epsilon)(\rho^{2}+|a|^{2}).
$$
Then, from Lemma \ref{Elementary inequality}, it follows
\begin{equation}\label{eq3-proof-Est lemma}
\begin{aligned}
|\gamma^{\epsilon}(s)|^{2}&\leq M^2\left(\rho^{2}+|Y^{\epsilon}(s-)|^{2}\right)\\
&\leq\frac{M^2}{\big(1- \epsilon\lambda(\epsilon)M\big)^{2}}
\cdot\left(\rho^{2}+|\tilde{Y}^{\epsilon}(s)|^{2}\right)\\
&\leq 2M^{2}\left(\rho^{2}+|\tilde{Y}^{\epsilon}(s)|^{2}\right) ,
\end{aligned}
\end{equation}
where the last inequality is obtained by $1-\epsilon\lambda(\epsilon)M>\frac{1}{\sqrt{2}}$ and thus for any $s\leq\tau$,
\begin{equation}\label{eq4-proof-Est lemma}
\begin{aligned}
&\Big|\langle\nabla\Phi(\tilde{Y}^{\epsilon}(s)\otimes\nabla\Phi(\tilde{Y}^{\epsilon}(s),
\gamma^{\epsilon}(s)\otimes\gamma^{\epsilon}(s)\rangle\Big|\\
&\leq\frac{4|\gamma^{\epsilon}(s)|^{2}\big|\tilde{Y}^{\epsilon}(s)\big|^{2}}
{\Big(\rho^{2}+\big|\tilde{Y}^{\epsilon}(s)\big|^{2}\Big)^{2}}
\leq \frac{8M^2\big|\tilde{Y}^{\epsilon}(s)\big|^{2}}{\rho^{2}+\big|\tilde{Y}^{\epsilon}(s)\big|^{2}}\leq 8M^2.
\end{aligned}
\end{equation}
Moreover, there exists $\theta'\in[0,1]$ such that
$$
\psi(\tilde{Y}^{\epsilon}(s))
=\psi(Y^{\epsilon}(s))
\exp\left(\frac{2\epsilon\theta}{\lambda(\epsilon)}
\cdot\frac{<Y^{\epsilon}(s-)+\theta'\theta\lambda(\epsilon)\epsilon\gamma^{\epsilon}(s), \gamma^{\epsilon}(s)>}{\rho^{2}+|Y^{\epsilon}(s-)+\theta'\theta\lambda(\epsilon)\epsilon\gamma^{\epsilon}(s)|^{2}}\right).
$$
Following the same line as in the proof of (\ref{eq3-proof-Est lemma}), we have for $s\in[0,\tau]$
$$
|\gamma^{\epsilon}(s)|^{2}\leq 2M^2\left(\rho^{2}+|Y^{\epsilon}(s-)+\theta'\theta\lambda(\epsilon)\epsilon\gamma^{\epsilon}(s)|^{2}\right).
$$
Therefore, we obtain
\begin{equation}\label{eq7-proof-Est lemma}
\psi(\tilde{Y}^{\epsilon}(s))\leq\psi(Y^{\epsilon}(s))e^{2\sqrt{2}M\epsilon/{\lambda(\epsilon)}},\quad s\leq\tau.
\end{equation}
Now, the fact $\nabla\otimes\nabla\Phi_{i,j}(u)=\frac{2\delta_{i,j}}{\rho^{2}+\|u\|^{2}}-\frac{4u_{i}u_{j}}{(\rho^{2}+\|u\|^{2})^{2}}$
gives the following operator norm
\begin{equation*}
\Big\|\nabla\otimes\nabla\Phi\big(\tilde{Y}^{\epsilon}(s)\big)\Big\|
\leq \frac{6}{\rho^{2}+\big|\tilde{Y}^{\epsilon}(s)\big|^{2}},
\end{equation*}
which together with~(\ref{eq3-proof-Est lemma}) implies
\begin{equation}\label{eq5-proof-Est lemma}
\Big|\langle\nabla\otimes\nabla\Phi\big(\tilde{Y}^{\epsilon}(s)\big),
\gamma^{\epsilon}(s)\otimes\gamma^{\epsilon}(s)\rangle\Big|
\leq 12M^2.
\end{equation}
Combing ~(\ref{eq4-proof-Est lemma}), \eqref{eq7-proof-Est lemma} and~(\ref{eq5-proof-Est lemma}), we have
\begin{equation}\label{eq6-proof-Est lemma}
\big|\mathcal{A}^{\epsilon}(s)\big|\leq \left(4M^2\epsilon^{2}\lambda^{-2}(\epsilon)+6M^2\epsilon^{2}\right)e^{2\sqrt{2}M\epsilon/{\lambda(\epsilon)}}\psi(Y^{\epsilon}(s)).
\end{equation}

Secondly, it holds for any $s\leq\tau$,
$$
\big|\langle\nabla\psi(Y^{\epsilon}(s-)),h^{\epsilon}(s)\rangle\big|\leq\frac{2B}{\lambda^2(\epsilon)}\psi(Y^{\epsilon}(s-)).
$$
Together with (\ref{eq-0-proof-Est lemma}) and (\ref{eq6-proof-Est lemma}), we can get for any $t\leq\tau$,
\begin{align*}
\psi(Y^{\epsilon}(t))
&\leq\psi(y^{\epsilon}_{0})+\int_0^t\mathcal{B}^{\epsilon}(s)\tilde{N}^{\epsilon^{-2}}(ds)\\
&\quad+\left(\Big(4M^2\epsilon^{2}\lambda^{-2}(\epsilon)+6M^2\epsilon^{2}\Big)e^{2\sqrt{2}M\epsilon/{\lambda(\epsilon)}}
+2B\lambda^{-2}(\epsilon)\right)\int^{t}_{0}\psi(Y^{\epsilon}(s))ds.
\end{align*}

Finally, letting~$\tau_{1}:=\inf\left\{t\geq0; |Y^{\epsilon}(t)|\geq\delta\right\}$, we have
\begin{align*}
&\mathbb{E}\psi(Y^{\epsilon}(\tau\wedge\tau_{1}\wedge t))\\
&\leq\psi(y^{\epsilon}_0)+\left(\Big(4M^2\epsilon^{2}\lambda^{-2}(\epsilon)+6M^2\epsilon^{2}\Big)e^{2\sqrt{2}M\epsilon/{\lambda(\epsilon)}}
+2B\lambda^{-2}(\epsilon)\right)
\int^{t}_{0}\psi(Y^{\epsilon}(\tau\wedge\tau_{1}\wedge s))ds,
\end{align*}
which implies by Gronwall inequality
$$
\mathbb{E}\psi(Y^{\epsilon}(\tau\wedge\tau_{1}))
\leq \psi(y^{\epsilon}_0)\exp\left\{\Big(4M^2\epsilon^{2}\lambda^{-2}(\epsilon)+6M^2\epsilon^{2}\Big)e^{2\sqrt{2}M\epsilon/{\lambda(\epsilon)}}
+2B\lambda^{-2}(\epsilon)\right\}.
$$
Therefore,
\begin{align*}
&\mathbb P\Big(\sup_{t\in[0,\tau]}|Y^{\epsilon}(t)|>\delta\Big)\\
&\leq \mathbb P\Big(\psi(Y^{\epsilon}(\tau\wedge\tau_{1}))\geq\psi(\delta)\Big)\\
&\leq\frac{\mathbb{E}\psi(Y^{\epsilon}(\tau\wedge\tau_{1}))}{\psi(\delta)}\\
&=\Big(\frac{\rho^{2}+|y^{\epsilon}_0|^{2}}{\rho^{2}+\delta^{2}}\Big)^{\frac{1}{\lambda^2(\epsilon)}}
\exp\left\{\Big(4M^2\epsilon^{2}\lambda^{-2}(\epsilon)+6M^2\epsilon^{2}\Big)e^{2\sqrt{2}M\epsilon/{\lambda(\epsilon)}}
+2B\lambda^{-2}(\epsilon)\right\},
\end{align*}
which completes the proof of this lemma.
\end{proof}

\subsection{Deviation inequalities for Poisson random integrals}

\begin{pro}\label{exponential inequality}
 For the stopping time $\tau$, assume that there exist some constants $\rho$, $\iota_1$ and $\iota_2$ such that
$$
\iota_1\leq\tau\leq\iota_2,\quad\sup_{s\in[\iota_1,\tau]}|\gamma(s)|\leq\rho.
$$

For any $a>0$ and $\eta>0$, we have
\begin{equation}\label{eq-lem-exponential inequality}
\begin{aligned}
 &P\Big(\sup_{t\in[\iota_1,\tau]}\lambda(\epsilon)\epsilon\int^{t}_{\iota_1}\gamma(s)\tilde{N}^{\epsilon^{-2}}(ds)>a\Big)\\
 &\leq \exp\left\{-\lambda^{-2}(\epsilon)\Big(\eta a-\frac{1}{2}\eta^2\rho^2(\iota_2-\iota_1)e^{\epsilon\lambda^{-1}(\epsilon)\eta\rho}\Big)\right\}.
\end{aligned}
\end{equation}
\end{pro}

\begin{proof}
 Define the stopping time
$$
\hat{\tau}=\inf\Big\{t\geq\iota_1; \lambda(\epsilon)\epsilon\int^{t}_{\iota_1}\gamma(s)\tilde{N}^{\epsilon^{-2}}(ds)\geq a\Big\}.
$$
It holds that
\begin{align*}
&\Big\{\sup_{t\in[\iota_1,\tau]}\lambda(\epsilon)\epsilon\int^{t}_{\iota_1}\gamma(s)\tilde{N}^{\epsilon^{-2}}(ds)>a\Big\}
\subset
\Big\{\lambda(\epsilon)\epsilon\int^{\tau\wedge\hat{\tau}}_{\iota_1}\gamma(s)\tilde{N}^{\epsilon^{-2}}(ds)\geq a\Big\}.
\end{align*}
For~$\eta>0$, Chebyshev inequality gives that
\begin{align*}
&\mathbb P\Big(\lambda(\epsilon)\epsilon\int^{\tau\wedge\hat{\tau}}_{\iota_1}\gamma(s)\tilde{N}^{\epsilon^{-2}}(ds)\geq a\Big)\\
&\leq\exp\Big\{-\lambda^{-2}(\epsilon)\eta a\Big\}
\mathbb E\exp\Big\{\frac{\epsilon\eta}{\lambda(\epsilon)}\int^{\tau\wedge\hat{\tau}}_{\iota_1}\gamma(s)\tilde{N}^{\epsilon^{-2}}(ds)\Big\}.
\end{align*}
Notice that for any~$p>1$
$$
\bigg\{\exp\Big\{\frac{p\epsilon\eta}{\lambda(\epsilon)}\int^{\tau\wedge\hat{\tau}\wedge t}_{\iota_1}\gamma(s)\tilde{N}^{\epsilon^{-2}}(ds)
-\epsilon^{-2}\int^{\tau\wedge\hat{\tau}\wedge t}_{\iota_1}
\big(e^{\frac{p\epsilon\eta}{\lambda(\epsilon)}\gamma(s)}-1-\frac{p\epsilon\eta}{\lambda(\epsilon)}\gamma(s)\big)ds\Big\}, t\geq0\bigg\}
$$
is an exponential martingale.
Therefore,  by H\"{o}lder's inequality, for $\frac{1}{p}+\frac{1}{q}=1$ with $p>0$, $q>0$,
\begin{align*}
&\mathbb E\exp\Big\{\frac{\epsilon\eta}{\lambda(\epsilon)}\int^{\tau\wedge\hat{\tau}}_{\iota_1}\gamma(s)\tilde{N}^{\epsilon^{-2}}(ds)\Big\}\\
&\leq\mathbb E^{\frac{1}{q}}\exp\Big\{\frac{q\epsilon^{-2}}{p}\int^{\tau\wedge\hat{\tau}}_{\iota_1}
\big(e^{\frac{p\epsilon\eta}{\lambda(\epsilon)}\gamma(s)}-1-\frac{p\epsilon\eta}{\lambda(\epsilon)}\gamma(s)\big)ds\Big\}\\
&\leq \exp\left\{\frac{1}{2}\lambda^{-2}(\epsilon)p\eta^2\rho^2(\iota_2-\iota_1)e^{\epsilon\lambda^{-1}(\epsilon)p\eta\rho}\right\},
\end{align*}
where the last inequality is derived from the fact that
$$
e^x-1-x\leq \frac{x^2}{2}e^{|x|},\quad x\in\mathbb{R}.
$$
Consequently, it holds that
\begin{align*}
 &P\Big(\sup_{t\in[\iota_1,\tau]}\lambda(\epsilon)\epsilon\int^{t}_{\iota_1}\gamma(s)\tilde{N}^{\epsilon^{-2}}(ds)>a\Big)\\
 &\leq \exp\left\{-\lambda^{-2}(\epsilon)\Big(\eta a-\frac{1}{2}p\eta^2\rho^2(\iota_2-\iota_1)e^{\epsilon\lambda^{-1}(\epsilon)p\eta\rho}\Big)\right\}.
\end{align*}
Letting $p\to1$, we can get (\ref{eq-lem-exponential inequality}) and complete the proof of this proposition.
\end{proof}

{\bf Acknowledgements:} Hui JIANG is supported by the Fundamental Research Funds for the Central Universities
(No.NS2022069) and National Natural Science Foundation of China(Grant NOs.11771209, 11971227). Lihu Xu is supported by National Natural Science Foundation of China No. 12071499, Macao S.A.R. grant FDCT 0090/2019/A2 and University of Macau grant MYRG2020-00039-FST.
Qingshan YANG is supported by National Natural Science Foundation of China(Grant NO.11401090, 11971097, 11971098) and the Fundamental Research Funds for the Central Universities(Grant NO.2412019FZ031).

\end{document}